\documentclass[11pt]{amsart}

\usepackage{epigamath}

\usepackage[english]{babel}

\numberwithin{equation}{section}

\usepackage{enumitem}
\usepackage[matrix,arrow]{xy}
\usepackage{tikz-cd}

\newtheorem{thm}{Theorem}[section]
\newtheorem{prop}[thm]{Proposition}
\newtheorem{lem}[thm]{Lemma}
\newtheorem{coro}[thm]{Corollary}

\theoremstyle{definition}
\newtheorem{defi}[thm]{Definition}

\theoremstyle{remark} 

\newtheorem{rem}[thm]{Remark}

\newcommand{\C}{\mathbb{C}}

\newcommand{\qq}{\mathbb{Q}}
\newcommand{\F}{\mathbb{F}}

\newcommand{\ZZ}{\mathbb{Z}}

\newcommand{\HH}{\mathrm{H}}

\newcommand{\oscr}{\mathscr{O}}

\newcommand{\Tr}{\mathrm{Tr}}

\newcommand{\DD}{\mathbf{D}}

\DeclareMathOperator{\im}{Im}
\DeclareMathOperator{\End}{End}
\DeclareMathOperator{\Spec}{Spec}

\DeclareMathOperator{\Spa}{Spa}

\DeclareMathOperator*{\colim}{colim}

\DeclareMathOperator{\ord}{ord}
\DeclareMathOperator{\naive}{naive}
\DeclareMathOperator{\can}{can}
\newcommand{\et}{\mathrm{\acute{e}t}}
\DeclareMathOperator{\un}{un}

\EpigaVolumeYear{6}{2022} \EpigaArticleNr{16} \ReceivedOn{February 18,
2020}
\InFinalFormOn{August 11, 2021}
\AcceptedOn{April 6, 2022}

\title{Higher Hida and Coleman theories on the modular curve}
\titlemark{Higher Hida and Coleman theories on the modular curve}

\author{George Boxer}
\email{george.boxer@ens-lyon.fr}
\address{Unit\'e de Math\'ematiques pures et appliqu\'ees,
Ecole normale sup\'erieure de Lyon,
46 all\'ee d'Italie,
69 364 Lyon Cedex 07,
France}

\author{Vincent Pilloni}
\email{vincent.pilloni@ens-lyon.fr}
\address{Unit\'e de Math\'ematiques pures et appliqu\'ees,
Ecole normale sup\'erieure de Lyon,
46 all\'ee d'Italie,
69 364 Lyon Cedex 07,
France}

\authormark{G.~Boxer and V.~Pilloni}

\AbstractInEnglish{We construct Hida and Coleman theories for the degree $0$ and $1$ cohomology of automorphic line bundles on the modular curve, and we define a $p$-adic duality pairing between the theories in degree $0$ and $1$. }



\acknowledgement{This research was supported by the ERC-2018-COG-818856-HiCoShiVa, and the first author was supported in part by the NSF postdoctoral fellowship DMS-1503047.
}

\begin{document}



\maketitle

\begin{prelims}

\DisplayAbstractInEnglish

\bigskip

\DisplayKeyWords

\medskip

\DisplayMSCclass







\end{prelims}


\newpage

\setcounter{tocdepth}{1}

\tableofcontents


\section{Introduction}

In the 80's, Hida introduced an ordinary projector on modular forms and he constructed $p$-adic families of ordinary modular forms \cite{MR868300}. In the 90's, Coleman developed the finite slope  theory \cite{1997InMat.127..417C} and Coleman and Mazur constructed the eigencurve \cite{MR1696469}. These theories have now been extended to higher dimensional Shimura varieties. 

Hida and Coleman theories combine two ideas.  The first is to restrict modular forms from the full modular curve to the ordinary locus and its neighborhoods. The additional structure on the universal $p$-divisible group on and near the ordinary locus (the canonical subgroups) is used to $p$-adically interpolate the sheaves of modular forms and their sections.  The second idea is to use the Hecke operators at $p$ to detect when a section of the sheaf of modular forms defined on (a neighborhood of) the ordinary locus comes from a classical modular form.

Until recently, Hida and Coleman theories had only been developed for degree $0$ coherent cohomology groups.  In  the recent works \cite{pilloniHidacomplexes}, \cite{BCGP}  we began to develop them further in order to study the higher coherent cohomology of automorphic vector bundles on the Shimura varieties for the group $\mathrm{GSp}_4$, and we are now convinced that Hida and Coleman theories should exist in all cohomological degrees for all Shimura varieties.

The purpose of the current work is to confirm this prediction in the simple setting  of modular curves and to construct Hida and Coleman theories for the degree $1$ cohomology groups. We actually construct in parallel the theories for degree $0$ and degree $1$ cohomology, as this  sheds some new light on the usual degree $0$ theory.  We also prove a $p$-adic Serre duality, which gives a perfect pairing between the theories in cohomological degree $0$ and $1$, but our constructions are independent of this pairing.  

Let us describe  the results we prove. Let $X \rightarrow \Spec\ZZ_p$ be the compactified modular curve of level $\Gamma_1(N)$, where $N \geq 3$ is an integer prime to $p$, and let $D$ be the boundary divisor. Let $X_1 \rightarrow \Spec\F_p$ be the special fiber  and $X_1^{\ord}$ be the ordinary locus.  Let $\omega$ be the modular line bundle and for a weight $k\in\ZZ$ write $\omega^k$ for $\omega^{\otimes k}$.

\begin{thm}[Hida's control theorem] There is a Hecke operator  $T_p$  acting on the cohomology groups $\mathrm{R}\Gamma(X_1, \omega^k)$, $\mathrm{R}\Gamma_c({X}_1^{\ord}, \omega^k)$, and $\mathrm{R}\Gamma({X}_1^{\ord}, \omega^k)$,  and  an associated ordinary projector $e(T_p)$. 
Moreover, we have quasi-isomorphisms  $$e(T_p) \mathrm{R}\Gamma(X_1, \omega^k) \longrightarrow e(T_p) \mathrm{R}\Gamma(X_1^{\ord}, \omega^k)~~~~\textrm{if $k \geq 3$}$$ and $$e(T_p) \mathrm{R}\Gamma_c(X^{\ord}_1, \omega^k) \longrightarrow e(T_p) \mathrm{R}\Gamma(X_1, \omega^k)~~~~\textrm{if $k \leq -1$}.$$
\end{thm}

The coherent cohomology with compact support appearing in the theorem has been defined by Hartshorne \cite{MR0297775}.  If $SS\subseteq X_1$ denotes the (reduced) supersingular divisor, then we have
$$\HH^i_c(X_1^{\ord},\omega^k)=\lim_n\HH^i(X_1,\omega^k(-nSS)),\quad \HH^i(X_1^{\ord},\omega^k)=\colim_n\HH^i(X_1,\omega^k(nSS)).$$

The proof of this theorem relies on a local analysis of the cohomological correspondence $T_p$ at supersingular points. The above theorem may also be viewed as a vanishing result: $e(T_p)\mathrm{R}\Gamma(X_1,\omega^k)$ is concentrated in degree $0$ if $k \geq 3$, and degree $1$ if $k \leq -1$, because the ordinary locus is affine.  Of course, this vanishing result holds true even without the ordinary projector, from the Riemann--Roch theorem and the Kodaira--Spencer isomorphism, however the first argument is well suited for generalizations to higher dimensional Shimura varieties.

Let $\Lambda = \ZZ_p[[\ZZ_p^\times]]$ be the Iwasawa algebra. Each integer $k \in \ZZ$ defines a character of $\ZZ_p^\times$, and a morphism $k : \Lambda \rightarrow \ZZ_p$. 
\begin{thm} There are two projective $\Lambda$-modules $M$ and $N$ carrying an action of the Hecke algebra of level prime to $p$, and there are canonical, Hecke-equivariant isomorphisms for all $k \geq 3$:
\begin{enumerate}
\item$M \otimes_{\Lambda, k} \ZZ_p = e(T_p) \HH^0(X, \omega^k)$,
\item  $N \otimes_{\Lambda, k} \ZZ_p = e(T_p) \HH^1(X, \omega^{2-k}(-D))$,
\end{enumerate}

Moreover, there is a perfect pairing  
$ M \times N \rightarrow \Lambda$ which interpolates the classical Serre duality pairing. 

\end{thm}

The modules $M$ and $N$ are obtained by considering the ordinary factor of the cohomology and cohomology with compact support of the ordinary locus, with values in an interpolation sheaf of $\Lambda$-modules. 

\bigskip

Let $X_0(p) \rightarrow \Spa(\qq_p, \ZZ_p)$ be the adic modular curve of level $\Gamma_1(N) \cap \Gamma_0(p)$. We have two quasi-compact opens $X_0(p)^m $ and $X_0(p)^{\et}$ inside $X_0(p)$ which are respectively the loci where the universal subgroup of order $p$ has multiplicative and \'etale reduction. We let $X_0(p)^{m,\dagger} $ and $X_0(p)^{\et, \dagger}$ be the corresponding dagger spaces \cite{MR1739729}.

\begin{thm}[Coleman's classicality theorem]  For all $k \in \ZZ$, there is a well defined Hecke operator $U_p$ which is compact and has non negative slopes on $\HH^i(X_0(p), \omega^k)$, $\HH^i(X_0(p)^{m, \dagger}, \omega^k)$, and $\HH^i_c(X_0(p)^{\et, \dagger}, \omega^k)$.
Moreover, the natural maps $($where the superscript $< \star $ means slope less than $\star$ for $U_p)$: 
\begin{enumerate}
\item $\HH^i(X_0(p), \omega^k)^{< k-1} \rightarrow \HH^i(X_0(p)^{m, \dagger}, \omega^k)^{<k-1}$,
\item $\HH^i_c(X_0(p)^{\et, \dagger}, \omega^k)^{< 1-k} \rightarrow \HH^i(X_0(p), \omega^k)^{<1-k}$
\end{enumerate}
are isomorphisms. 
\end{thm}

The proof of this theorem is based on some simple estimates for the operator $U_p$ on the ordinary locus, reminiscent of \cite{MR2219265}. We can again deduce a vanishing theorem for the small slope classical cohomology (without appealing to the Riemann--Roch theorem).  

Coleman and Mazur constructed the eigencurve $\mathcal{C}$ of tame level $\Gamma_1(N)$. It carries a weight morphism $w: \mathcal{C} \rightarrow \mathcal{W}$ where $\mathcal{W}$ is the analytic adic space over $\Spa(\qq_p, \ZZ_p)$ associated with the Iwasawa algebra $\Lambda$. 
\begin{thm} The  eigencurve carries two coherent sheaves $\mathcal{M}$ and $\mathcal{N}$ interpolating the degree $0$ and $1$ finite slope cohomology. For any $k \in \ZZ$, we have 
\begin{enumerate} 
\item $\mathcal{M}_k^{< k-1} = \HH^0(X_0(p), \omega^k)^{< k-1}$,
\item $\mathcal{N}_k^{< k-1} = \HH^1(X_0(p), \omega^{2-k}(-D))^{< k-1}$,
\end{enumerate}
and there is a perfect pairing between $\mathcal{M}$ and $\mathcal{N}$, interpolating the usual Serre duality pairing. 
\end{thm}

\subsection{Acknowledgements}
We would like to thank Frank Calegari and Toby Gee for conversations about the topics in this paper, and David Loeffler for his comments.  We thank Ben Heuer and the anonymous referee for their corrections and suggestions. 

\section{Preliminaries}
\subsection{Finite flat cohomological   correspondences}
In this section all schemes are Noetherian.  We write $\mathrm{Coh}(X)$ for the category of coherent sheaves on a scheme $X$.
\subsubsection{The functors $f_\star$, $f^\star$ and $f^!$}
Let $f:X \rightarrow Y$ be a finite flat morphism of schemes.  We have a  functor $f_\star : \mathrm{Coh}(X) \rightarrow \mathrm{Coh}(Y)$.  It induces an equivalence of categories between $\mathrm{Coh}(X)$ and the category of coherent sheaves of $f_\star\oscr_X$-modules over $Y$.

The functor $f_\star$  has a left adjoint $f^\star : \mathrm{Coh}(Y) \rightarrow \mathrm{Coh}(X)$, given by $f_\star f^\star\mathscr{F}=\mathscr{F} \otimes_{\oscr_{Y}} f_\star \oscr_{X}$, as well as a right adjoint $f^! : \mathrm{Coh}(Y) \rightarrow \mathrm{Coh}(X)$ given by $f_\star f^! \mathscr{F} = \underline{\mathrm{Hom}}_{\oscr_Y}( f_\star \oscr_{X}, \mathscr{F})$.  For any $\mathscr{F} \in \mathrm{Coh}(X)$, we have an isomorphism  $f^! \mathscr{F} = f^! \oscr_{Y} \otimes_{\oscr_{X}} f^\star\mathscr{F}$. 

\medskip

A finite flat morphism $f : X \rightarrow Y$ has a trace map $\mathrm{tr}_f : f_\star \oscr_{X} \rightarrow \oscr_{Y}$. This trace is by definition a global section of $f^! \oscr_{Y}$ or equivalently a morphism $f^\star \oscr_{Y} \rightarrow f^! \oscr_{Y}$.  It follows that the trace map provides a natural transformation $f^\star \Rightarrow f^! $. 

\medskip

A finite flat morphism $ f: X \rightarrow Y$ is called Gorenstein if $f^! \oscr_X$ is an invertible sheaf. If $f : X \rightarrow Y$ if a local complete intersection morphism, then it is Gorenstein (see \cite[Corollary~21.19]{MR1322960}).

\subsubsection{Cohomological  correspondences}
\begin{defi} A finite flat correspondence over  a scheme $X$ is a scheme $C$ equipped with two finite flat morphisms  $X \stackrel{p_2}\leftarrow   C \stackrel{p_1}\rightarrow X$.
\end{defi}

\begin{defi} Let $\mathscr{F}$ be a coherent sheaf on $X$. A  finite flat cohomological correspondence  for $\mathscr{F}$ is the data of a pair $(C,T)$ consisting of a   finite flat correspondence $C$ and a map $ T : p_2^\star \mathscr{F} \rightarrow p_1^! \mathscr{F}$. 
\end{defi}

Given a finite flat cohomological correspondence $(C,T)$ on $\mathscr{F}$, we get a morphism  in cohomology that we also denote by $T$:
$$T : \mathrm{R}\Gamma(X, \mathscr{F}) \xrightarrow{ ~p_2^\star~ } \mathrm{R}\Gamma(C, p_2^\star \mathscr{F}) \xrightarrow{ ~T~ }  \mathrm{R}\Gamma(C, p_1^! \mathscr{F}) \xrightarrow{~\mathrm{tr}_{p_1}} \mathrm{R}\Gamma(X, \mathscr{F}).$$ 

\subsubsection{Restriction} Let $X$ be a scheme and $Y \hookrightarrow X$ be a closed subscheme defined by a sheaf of ideals $\mathscr{I} \subseteq \oscr_{X}$. For any quasi-coherent sheaf $\mathscr{F}$ over $X$, we let $$\underline{\Gamma}_Y(\mathscr{F}) = \mathrm{Ker} (\mathscr{F} \rightarrow \underline{\mathrm{Hom}}( \mathscr{I}, \mathscr{F}))$$ be the subsheaf of sections with scheme theoretic support in $Y$. 

\begin{prop}\label{prop-restriction-simple} Consider a  commutative diagram of schemes: 
\[
\xymatrix{ Y \ar[r]^{i} \ar[rd]_g & X \ar[d]^f \\
& Z}
\]
where $i$ is a closed immersion and $g$, $f$ are finite flat morphisms. 
Let $\mathscr{F}$ be a coherent sheaf on $Z$. Then $i_\star g^! \mathscr{F} = \underline{\Gamma}_Y f^! \mathscr{F}$.
\end{prop}

\begin{proof} Let us denote by $\mathscr{I}$ the ideal sheaf of $Y$ in $X$. The claim follows from the exact sequence:
$$ 0 \longrightarrow  \underline{\mathrm{Hom}}( g_\star \oscr_{Y}, \mathscr{F}) \longrightarrow   \underline{\mathrm{Hom}}( f_\star \oscr_{X}, \mathscr{F}) \longrightarrow   \underline{\mathrm{Hom}}( f_\star \mathscr{I}, \mathscr{F}).$$
\end{proof}

\subsection{Local finiteness and ordinary projectors} Let $R$ be a finite Artinian ring.

\begin{lem}\label{lem-existence-proj} Suppose $M$ is a finite $R$-module and $T\in\End_R(M)$.  The sequence $(T^{n!})_{n \in \mathbb{N}}$ is eventually constant and converges to an idempotent $e(T) \in \mathrm{End}_{R}(M)$.
\end{lem}
\begin{proof} We have a decreasing sequence of submodules $T^n(M)$ of $M$ which is eventually stationary since $M$ is artinian. Let $M_0$ be the limit.  $T$ restricts to a bijection on the finite set $M_0$, and hence for $n$ large enough $T^{n!}$ is the identity on $M_0$.  It follows that for all $n$ large enough, $T^{n!}$ is an idempotent projector from $M$ to $M_0$. 
\end{proof}

\begin{defi} Let $M$ be an $R$-module and let $T \in \mathrm{End}_{R}(M)$. We say that $T$ is locally finite on $M$ if $M$ is a union of finite $R$-modules which are stable under $T$. 
\end{defi}

\begin{rem} This is equivalent to: for any finite submodule $V \subseteq M$ there exists an $n>0$ such that $T^n V \subseteq \sum_{i=0}^{n-1} T^i V$. 
\end{rem}

Let $M$ be an $R$-module and let $T\in\mathrm{End}_R(M)$ be locally finite.  Then for each $v\in M$, Lemma~\ref{lem-existence-proj} implies that the sequence $T^{n!}v$ is eventually constant, and we define $e(T):M\to M$ by $e(T)v=T^{n!}v$ for $n$ sufficiently large.

Lemma~\ref{lem-existence-proj} further implies that $e(T)$ is an $R$-linear idempotent which commutes with $T$, and moreover $T$ is an isomorphism on $e(T)M$ and locally nilpotent on $(1-e(T))M$, in the sense that for each $v \in (1-e(T))M$, there exists $n>0$ such that $T^n v = 0$. 

If $f:M \rightarrow N$ is a morphism of $R$-modules, equivariant for a locally finite endomorphism $T$ (\textit{i.e.} we have locally finite endomorphisms $T\in\mathrm{End}_R(M)$ and $T\in\mathrm{End}_R(N)$ with $fT=Tf$) then we have $e(T)f=fe(T)$ and so $f(e(T)M)\subseteq e(T)N$ and $f((1-e(T))M)\subseteq (1-e(T))N$.

\begin{prop}\label{prop-disc-exact} Let $0 \rightarrow M \rightarrow N \rightarrow L \rightarrow 0$ be a short exact sequence of $R$-modules with an equivariant endomorphism $T$.
\begin{enumerate}
\item $T$ is locally finite on $N$ if and only if $T$ is locally finite on $M$ and $L$. 
\item If this holds, then $0 \rightarrow e(T) M \rightarrow e(T)N \rightarrow e(T)L \rightarrow 0$ is exact.
\end{enumerate}
\end{prop}

\begin{proof} The only point that requires a proof is that if $T$ is locally finite on $M$ and $L$, then it is locally finite on $N$.  Let $V \subseteq N$ be a finite submodule. Since $T$ is locally finite on $L$, there exists $n \geq 1$ such that $ T^n(V) \subseteq \sum_{i=0}^{n-1} T^iV + M$.  There is furthermore a finite submodule $W\subseteq M$ such that $$T^n(V)\subseteq\sum_{i=0}^{n-1}T^iV+W$$

Since $T$ is locally finite on $M$, there exists $m\geq 1$ such that:
$$ T^m(W) \subseteq \sum_{i=0}^{m-1}T^{i} W.$$
We finally conclude that $T^{n+m}(V+W) \subseteq \sum_{i=0}^{n+m-1} T^i(V +W)$, so that $\sum_{i=0}^{n+m-1} T^i (V +W)$ is a finite $R$-module, stable by $T$, containing $V$. 
\end{proof}

\bigskip

We will also need to work with a different notion of local finiteness, defined on certain topological $R$-modules.

\begin{defi}
A topological $R$-module $M$ is a profinite $R$-module if it is homeomorphic to a cofiltered limit $\lim_{i\in I}M_i$ of finite $R$-modules, or equivalently, if it is Hausdorff and has a basis of neighborhoods $\{N_i\}_{i\in I}$ of 0 consisting of open cofinite submodules.
\end{defi}

\begin{defi} Let $M$ be a profinite $R$-module, and let $T \in \mathrm{End}_{R}(M)$ be a continuous endomorphism. We say that $T$ is locally finite on $M$ if $M$ has a basis of neighborhoods of $0$ consisting of submodules $\{N_i\}$ such that  $T(N_i) \subseteq N_i$.
\end{defi}

\begin{rem} An endomorphism $T\in\mathrm{End}_R(M)$ is locally finite if and only if for any open submodule $V\subseteq M$, $\cap_{i=0}^\infty T^{-i}(V)\subseteq M$ is open.  Equivalently, it is locally finite if and only if for every open submodule $V\subseteq M$, there is some $n>0$ such that $\cap_{i=0}^{n-1}T^{-i}(V)\subseteq T^{-n}(V)$ (indeed, if $\cap_{i=0}^\infty T^{-i}(V)\subseteq M$ is open, then it is cofinite, and hence must be $\cap_{i=0}^{n-1} T^{-i}(V)$ for some $n>0$.)
\end{rem}

Let $M$ be a profinite $R$-module and let $T\in\mathrm{End}_R(M)$ be a locally finite endomorphism.  Pick a neighborhood basis $\{N_i\}_{i\in I}$ of 0 with $T(N_i)\subseteq N_i$.  Then for all $v\in M$ and each $i\in I$, Lemma~\ref{lem-existence-proj} implies that the sequence $T^{n!}v$ is eventually constant in $M/N_i$.  It follows that the sequence $T^{n!}v$ converges in $M$, to a limit that we denote $e(T)v$.

Lemma~\ref{lem-existence-proj} further implies that $e(T)$ is a continuous, $R$-linear idempotent which commutes with $T$, and moreover $T$ is an isomorphism on $e(T)M$ and topologically nilpotent on $(1-e(T))M$, in the sense that for each $v \in (1-e(T))M$, $\{T^nv\}_{n\in\mathbb{N}}$ converges to 0.

If $f:M \rightarrow N$ is a continuous morphism of profinite $R$-modules, equivariant for a locally finite endomorphism $T$ then we have $e(T)f=fe(T)$ and so $f(e(T)M)\subseteq e(T)N$ and $f((1-e(T))M)\subseteq (1-e(T))N$.

 \begin{prop}\label{prop-pro-exact} Let $0 \rightarrow M \rightarrow N \rightarrow L \rightarrow 0$ be a short exact sequence of profinite $R$-modules with an equivariant continuous endomorphism $T$. 
\begin{enumerate}
\item $T$ is locally finite on $N$ if and only if $T$ is locally finite on $M$ and $L$. 
\item If this holds, then $0 \rightarrow e(T) M \rightarrow e(T)N \rightarrow e(T)L \rightarrow 0$ is exact.
\end{enumerate}
\end{prop}
\begin{proof} The only point that requires a proof is that if $T$ is locally finite on $M$ and $L$, then it is locally finite on $N$.  Let $V \subset N$ be an open submodule. 
Since $T$ is locally finite on $M$, we deduce that there is $n \geq 1$ such that $\cap_{i=0}^{n-1} T^{-i} (V\cap M) \subseteq T^{-n}(V \cap M)$.  It follows that:
$$ (T^{-n} V + M) \bigcap \cap_{i=0}^{n-1} T^{-i} (V) \subseteq T^{-n}(V).$$
To shorten notations, let us denote by $W = T^{-n} V +M$. Since $W$ is open in $N$, its image $\overline{W}$ in $L$ is open (it is closed and finite index.) Since $T$ is locally finite on $L$, there is $m \geq 0$ such that $\cap_{i=0}^{m-1} T^{-i} (\overline{W}) \subseteq T^{-m}(\overline{W}) $. We deduce that $\cap_{i=0}^{m-1} T^{-i} ({W}) \subseteq T^{-m}({W}) $.
It follows that: 
\begin{eqnarray*}
T^{-(n+m)} (V \cap W) &=& T^{-m}(T^{-n} V) \cap T^{-(n+m)}(W) \\
&\supseteq & T^{-m} (W\bigcap \cap_{i=0}^{n-1} T^{-i} (V)) \bigcap \cap_{i=0}^{m-1} T^{-i} W \\
& \supseteq & T^{-m} (W) \bigcap \cap_{i=0}^{m+n-1} T^{-i} (V) \bigcap \cap_{i=0}^{m-1} T^{-i} W \\
& \supseteq &  \cap_{i=0}^{m+n-1} T^{-i} (V) \bigcap \cap_{i=0}^{m-1} T^{-i} W\\
& \supseteq &  \cap_{i=0}^{m+n-1} T^{-i} (V\cap W) 
\end{eqnarray*}
Therefore, $\cap_{i=0}^{m+n-1} T^{-i} (V\cap W)$ is an open submodule of $V$ which is stable under $T$.
\end{proof}

We will also consider the situation that $R=\lim_n R/\mathfrak{m}^n$ is a complete, semilocal, noetherian ring with Jacobson radical $\mathfrak{m}$, and $R/\mathfrak{m}$ is finite (\textit{e.g.} $R=\ZZ_p$ or $\ZZ_p[[\ZZ_p^\times]]$).  If $M=\lim_n M/\mathfrak{m}^nM$ is a complete $R$-module with an endomorphism $T$, we say that $T$ is locally finite if for all $n$ (or equivalently for a single $n$ by Propositions~\ref{prop-disc-exact} and~\ref{prop-pro-exact}) the induced endomorphism $T$ of $M/\mathfrak{m}^nM$ is locally finite in one of the two senses considered above.  In this case the idempotent endomorphisms $e(T)$ of $M/\mathfrak{m}^nM$ are compatible and define an idempotent endomorphism $e(T)$ of $M$.

\section{The cohomological correspondence $T_p$}

Let $N\geq 3$ be an integer and let $p$ be a prime.  Let $X \rightarrow \Spec\ZZ_p$ be the compactified  modular curve of level $\Gamma_1(N)$ (\cite{DR}). This is a proper smooth relative curve.  Denote by $D$ the boundary divisor, and by $E \rightarrow X$ the semi-abelian scheme which extends the universal elliptic curve and denote by $e$ the unit section.  We let $\omega_{E} = e^\star \Omega^1_{E/X}$.  For any $k \in \ZZ$, we denote by  $\omega^k = \omega_E^{\otimes k}$.

\subsection{The cohomological correspondences $T_p$}\label{section-T_p} 
We denote by $p_1, p_2 : X_0(p) \rightarrow X$ the Hecke correspondence which parametrizes an isogeny  $\pi : p_1^\star E \rightarrow p_2^\star E$  of degree $p$ (compatible with the $\Gamma_1(N)$ level structure). We denote by $D_0(p)$ the boundary divisor in $X_0(p)$ (which is reduced, so that $D_0(p) = (p_i^{-1} D)_{\mathrm{red}}$).  We let $\pi_k : p_2^\star \omega^k \dashrightarrow p_1^\star \omega^k$ be the rational map  which is deduced from the pull-back map on differentials $p_2^\star \omega_E \rightarrow p_1^\star \omega_E$ (this map is regular if $k \geq 0$, and is an isomorphism over $\qq_p$ for all $k$). We also denote by $\pi_k^{-1} : p_1^\star \omega^k \dashrightarrow p_2^\star \omega^k$ the inverse of $\pi_k$. 
We also have a dual isogeny $\pi^\vee : p_2^\star E \rightarrow p_1^\star E$ and we denote by $\pi^\vee_k : p_1^\star \omega^k \dashrightarrow p_2^\star \omega^k$ the rational map which is deduced from the pull-back map on differentials $p_1^\star \omega_E \rightarrow p_2^\star \omega_E$. We also denote by $(\pi^\vee_k)^{-1} : p_2^\star \omega^k \dashrightarrow p_1^\star \omega^k$ the inverse of $\pi^\vee_k$.  We have the following formula relating $\pi_k$ and $\pi^\vee_k$:
$$ \pi_k  \circ \pi^\vee_k = p^k \mathrm{Id}.$$

We have natural trace maps $ \mathrm{tr}_{p_1} : \oscr_{X_0(p)} \rightarrow p_1^! \oscr_{X}$ and $ \mathrm{tr}_{p_2} : \oscr_{X_0(p)} \rightarrow p_2^! \oscr_{X}$.  Note that since $p_1$ and $p_2$ are local complete intersection morphisms, $p_1^! \oscr_X$ and $p_2^! \oscr_X$ are invertible sheaves. 

\begin{lem} The restrictions of $\mathrm{tr}_{p_1}$ and $\mathrm{tr}_{p_2}$ to $\oscr_{X_0(p)}(-D_0(p))$ factor through maps $ \mathrm{tr}_{p_1} : \oscr_{X_0(p)}(-D_0(p)) \rightarrow p_1^! (\oscr_{X}(-D))$ and $ \mathrm{tr}_{p_2} : \oscr_{X_0(p)}(-D_0(p)) \rightarrow p_2^! (\oscr_{X}(-D))$.
\end{lem}

\begin{proof} The boundary divisors in $X$ and $X_0(p)$ are reduced. The lemma boils down to the statement that the trace of a function vanishing along the boundary on $X_0(p)$ vanishes on the boundary on $X$.
\end{proof}

\medskip

We have a ``naive'' (\textit{i.e.} unnormalized) cohomological correspondence:
\[\begin{tikzcd}[column sep=normal]
T^{\naive}_{p,k} :  p_2^\star \omega^k \ar[r,dashed] & p_1^! \omega^k
\end{tikzcd}\]
which is the rational map defined by taking the tensor product of the map $\pi_k:p_2^\star \omega^k \dashrightarrow p_1^\star \omega^k$ and the map $ \mathrm{tr}_{p_1} : \oscr_{X_0(p)} \rightarrow p_1^! \oscr_{X}$, and similarly a map $ T^{\naive}_{p,k} :  p_2^\star (\omega^k(-D) )\dashrightarrow p_1^! (\omega^k(-D))$ defined using $p_2^\star(\omega^k(-D))=(p_2^\star\omega^k)(-p_2^{-1}(D))\subseteq (p_2^\star\omega^k)(-D_0(p))$.  We finally let $T_{p,k} = p^{-\inf\{1,k\}} T^{\naive}_{p, k}$. 

\begin{prop}\label{prop-T_p-constructed} $T_{p, k}$ is a cohomological correspondence $p_2^\star \omega^k  \rightarrow p_1^! \omega^k$. 
\end{prop}

\begin{proof} The map $T_{p, k}$ is a rational map between invertible sheaves over the regular scheme $X_0(p)$ so we can check that it is defined outside of codimension 2.  Since the map is defined over $\qq_p$, we can thus localize at a generic point $\xi$ of the special fiber and we are left to prove that it is well defined locally at these points. There are two types of generic points  corresponding to the possibility that the isogeny $p_1^\star E \rightarrow p_2^\star E$ is either multiplicative or \'etale. Let us first assume that $\xi$ is \'etale. 
The differential  map $(p_2^\star \omega)_{\xi} \tilde{\rightarrow} (p_1^\star \omega)_{\xi}$ is an isomorphism and the map $(tr_{p_1})_\xi : (p_1^\star \oscr_{X})_\xi \rightarrow (p_1^! \oscr_{X})_{\xi}$ factors into an isomorphism: $(tr_{p_1})_\xi : (p_1^\star \oscr_{X})_\xi \tilde{\rightarrow} p(p_1^! \oscr_{X})_{\xi}$.
It follows that $ (T^{\naive}_{p,k})_\xi  :  (p_2^\star \omega^k)_\xi \tilde{\rightarrow} p (p_1^! \omega^k)_\xi$.
Let us next assume that $\xi$ is multiplicative. The differential  map $(p_2^\star \omega)_{\xi} \tilde{\rightarrow} (p_1^\star \omega)_{\xi}$ factors into an isomorphism  $(p_2^\star \omega)_\xi \tilde{\rightarrow} p (p_1^! \omega)_\xi$ and the map $(tr_{p_1})_\xi : (p_1^\star \oscr_{X})_\xi \tilde{\rightarrow} (p_1^! \oscr_{X})_{\xi}$ is an isomorphism. It follows that $ (T^{\naive}_{p,k})_\xi  :  (p_2^\star \omega^k)_\xi \tilde{\rightarrow} p^k (p_1^! \omega^k)_\xi$. We deduce that $T_{p,k}$ is indeed a well defined map and that it is optimally integral. 
\end{proof} 

\bigskip

When the weight is clear, we often write $T_{p}$. The map $T_{p}$ induces a map on cohomology: 
$$ T_p \in \mathrm{End}\mathrm{R\Gamma}(X, \omega^k)~~~\textrm{and} ~~~\mathrm{End}\mathrm{R\Gamma}(X, \omega^k(-D))$$
obtained by composing the maps: 
$$ \mathrm{R}\Gamma(X, \omega^k) \xrightarrow{~p_2^\star~} \mathrm{R}\Gamma(X_0(p), p_2^\star\omega^k) \xrightarrow{~T_p~} \mathrm{R}\Gamma(X_0(p), p_1^!\omega^k) \xrightarrow{~\mathrm{tr}_{p_1}} \mathrm{R}\Gamma(X, \omega^k)$$
and similarly for cuspidal cohomology.

\begin{rem} Proposition~\ref{prop-T_p-constructed} is a particular instance of constructions performed in \cite{F-Pilloni}, where the problem of constructing Hecke operators on the integral coherent cohomology of more general Shimura varieties is considered. 
\end{rem}

\begin{rem}\label{rem-$q$-expansion} One can check that our map $T_{p, k}$ has the following effect on $q$-expansions (of given Nebentypus $\chi : \ZZ/N\ZZ^\times \rightarrow \overline{\ZZ}_p^\times$): it maps $\sum a_n q^n $ to $\sum a_{np} q^n + p^{k-1} \chi(p) \sum a_n q^{np}$ if $k \geq 1$ and to $p^{1-k}\sum a_{np} q^n +  \chi(p)\sum a_n q^{n p}$  if $k \leq 1$.
\end{rem}
\begin{rem} Our normalization is consistent with standard conjectures on the existence and properties  of Galois representations associated to automorphic forms (\cite{MR3444225}). The cohomology groups $\mathrm{H}^i(X, \omega^k) \otimes \C$ can be computed using automorphic forms and for any $\pi = \pi_\infty \otimes \pi_f$ contributing to the cohomology, we find that the infinitesimal character of $\pi_\infty$ is  given by $(t_1, t_2) \mapsto  t_1^{\frac{1}{2}} t_2^{k- \frac{1}{2}}$ and is indeed $C$-algebraic. By the Satake isomorphism, we know that $T^{\naive}_p \vert \pi_p = p^{1/2} \mathrm{Trace} ( \mathrm{Frob}^{-1}_p \vert \mathrm{rec} (\pi_p))$ (because $T_p^{\naive}$ corresponds to the co-character $t \mapsto (1, t^{-1})$  \cite[Rem. 5.6]{F-Pilloni}!). 
It is convenient to introduce the twist $\pi \otimes \vert \cdot\vert^{-\frac{1}{2}}$ which is $L$-algebraic, for which we find that the infinitesimal character of $\pi_\infty \otimes \vert \cdot\vert^{-\frac{1}{2}}$  is $(t_1, t_2) \mapsto   t_2^{k- 1}$. We make the following normalizations: the Hodge-Tate weight of the cyclotomic character is $-1$, and we normalize the reciprocity law by using geometric Frobenii. With these conventions,  the Hodge cocharacter is $t \mapsto ( 1, t^{1- k})$ and the corresponding Hodge polygon has slopes $1-k$ and $0$. We find that $T^{\naive}_p\vert \pi_p = p \mathrm{Trace} \big( \mathrm{Frob}^{-1}_p \vert \mathrm{rec} (\pi_p \otimes \vert\cdot \vert^{-\frac{1}{2}})\big)$.   
The Katz-Mazur inequality predicts that the Newton polygon (which has slopes the $p$-adic valuations of two eigenvalues of $\mathrm{Frob}_p$) is above the Hodge polygon with same ending and initial point, from which we find that $v(T^{\naive}_p \vert \pi_p) \geq  \inf\{1, k\}$ and that $T_p$ is indeed optimally integral. 
\end{rem}

 \subsection{Duality} We let $\DD_{\ZZ_p}  = \mathrm{RHom}_{\ZZ_p}( - , \ZZ_p)$ be the dualizing functor on the category of bounded complexes of  finite type $\ZZ_p$-modules \cite[Chapter~V]{Hartshorne}.  We let $\omega_{X/\ZZ_p}$ and $\omega_{X_0(p)/\ZZ_p}$ be the dualizing modules.  We recall that $\omega_{X/\ZZ_p}=\Omega^1_{X/\ZZ_p}$, while $\omega_{X_0(p)/\ZZ_p}=j_\star\Omega^1_{U/\ZZ_p}$, where $j:U\to X_0(p)$ is the complement of the supersingular locus in the special fiber, which is also the smooth locus of $X_0(p)\to \Spec(\ZZ_p)$.  Then we have dualizing functors $\DD_X=\mathrm{R\underline{Hom}}(- ,\omega_{X/\ZZ_p}[1])$ and $\DD_{X_0(p)}=\mathrm{R\underline{Hom}}(-,\omega_{X_0(p)/\ZZ_p}[1])$ on the derived category of bounded complexes of coherent sheaves on $X$ and $X_0(p)$. When the context is clear we only write $\DD$ for any of these dualizing functors.
 

We have the following Serre duality isomorphism  (\cite[Chapter~III, Theorem~11.1]{Hartshorne}):  $$\DD_{\ZZ_p}(\mathrm{R}\Gamma(X,\omega^k)) = \mathrm{R}\Gamma(X,\DD_X(\omega^k))$$ and similarly for cuspidal cohomology.  

We now want to understand how this duality isomorphism behaves with respect to Hecke operators. The Hecke $T_p$ is defined as a composition 
  $$ \mathrm{R}\Gamma(X, \omega^k) \xrightarrow{~p_2^\star~} \mathrm{R}\Gamma(X_0(p), p_2^\star \omega^k) \xrightarrow{~T_p~} \mathrm{R}\Gamma(X_0(p), p_1^!\omega^k) \xrightarrow{~\mathrm{tr}_{p_1}} \mathrm{R}\Gamma(X, \omega^k)$$
and hence dualizes to a composition: 
  $$\DD(\mathrm{R}\Gamma(X, \omega^k)) \xrightarrow{~p_1^\star~} \DD(\mathrm{R}\Gamma(X_0(p), p_1^!\omega^k) ) \xrightarrow{\DD(T_p)} \DD(\mathrm{R}\Gamma(X_0(p), p_2^\star\omega^k)) \xrightarrow{~\mathrm{tr}_{p_2}} \DD(\mathrm{R}\Gamma(X, \omega^k)).$$
  
  We have $$\DD(\mathrm{R}\Gamma(X_0(p), p_1^!\omega^k) ) = \mathrm{R}\Gamma((X_0(p), p_1^\star \DD(\omega^k)),$$ $$\DD(\mathrm{R}\Gamma(X_0(p), p_2^\star\omega^k) ) = \mathrm{R}\Gamma((X_0(p), p_2^! \DD(\omega^k)),$$
according to \cite[Chapter~III, Theorem~11.1 and Chapter~V, Proposition~8.5]{Hartshorne}, 
  and it remains to understand $\DD(T_p) : p_1^\star \DD(\omega^k) \rightarrow p_2^! \DD(\omega^k)$. We first  recall that we have the Kodaira--Spencer isomorphism over $X$ \cite[A.1.3.17]{MR0447119} $$KS : \omega^2 (-D) \rightarrow \Omega^1_{X/\ZZ_p}.$$
  We consider  the correspondence $T_p  : p_2^\star \omega^k \rightarrow p_1^! \omega^k$. Applying the functor $\DD_{X_0(p)}$ and making a $[-1]$-shift yields a map $\DD(T_p)  : p_1^\star (\omega^{-k} \otimes \omega_{X/\ZZ_p})  \rightarrow p_2^! (\omega^{-k} \otimes \omega_{X/\ZZ_p})$. 
If we use the Kodaira--Spencer isomorphism on both sides, we obtain a map: 
$\DD(T_p)  : p_1^\star (\omega^{-k+2}(-D))  \rightarrow p_2^! (\omega^{-k+2}(-D))$.

We can also consider the transpose correspondence $T_p^t:p_1^\star(\omega^k(-D))  \rightarrow p_2^! (\omega^k(-D))$, defined as $p^{-\inf\{1,k\}}T_p^{t,\naive}$, where $T_p^{t,\naive}$ is the tensor product of $\pi_k^\vee:p_1^\star\omega^k\dashrightarrow p_2^\star\omega^k$ and $tr_{p_1}:\oscr_{X_0(p)}(-D_0(p))\to p_1^!(\oscr_X(-D))$.  We can consider the diamond operator $\langle p\rangle$ on $X$ and $X_0(p)$ which multiplies the $\Gamma_1(N)$ level structure by $p$, as well as the Atkin--Lehner map $w:X_0(p)\to X_0(p)$, which sends the isogeny $\pi:p_1^\star E\to p_2^\star E$ to the dual isogeny $\pi^\vee:p_2^\star E\to p_1^\star E$.  We have the formulas $w^2=\langle p\rangle$, $p_1w=p_2$, and $p_2w=\langle p\rangle p_1$, and pulling back by $w$, we see that $T_p^t$ acts on cohomology as $\langle p\rangle^{-1}T_p$.

The main result of this section is now the following: 
\begin{prop}\label{prop-T_p-selfdual} We have $\DD(T_p) =  T_p^t$ as maps $p_1^\star(\omega^{-k+2}(-D))\to p_2^!(\omega^{-k+2}(-D))$, and hence we have $\DD(T_p)=\langle p\rangle^{-1} T_p$ as endomorphisms of $\DD_{\ZZ_p}(\mathrm{R}\Gamma(X,\omega^k))=\mathrm{R}\Gamma(X,\omega^{2-k}(-D))$.
\end{prop}

  \begin{lem}\label{lem-KS} The following diagram is commutative: 
\[
  \xymatrix{  & \omega_{X_0(p)/\ZZ_p}(D_0(p)) & \\
  p_2^\star \Omega^{1}_{X/\ZZ_p}(D)  \ar[ur]^{\mathrm{tr}_{p_2}} & &  p_2^\star \Omega^{1}_{X/\ZZ_p}(D) \ar[ul]_{\mathrm{tr}_{p_1}} \\
  p_2^\star \omega^2 \ar[u]^{p_2^\star KS} \ar[rr]^{  p^{-1}\pi_2} &&  p_1^\star \omega^2 .\ar[u]_{p_1^\star KS}}
\]
  \end{lem}
\begin{proof} Over $X_0(p)$ we have a map $\pi^\star : p_2^\star( \mathcal{H}^1_{dR} (E/X), \nabla) \rightarrow p_1^\star( \mathcal{H}^1_{dR} (E/X), \nabla)$ which induces a commutative diagram: 
\[
  \xymatrix{  \omega_{X_0(p)/\ZZ_p}(D_0(p))  \otimes p_2^\star \omega^{-1} \ar[rr]^{1 \otimes (\pi^\vee_{-1})^{-1}} & & \omega_{X_0(p)/\ZZ_p}(D_0(p)) \otimes p_1^\star \omega^{-1}   \\
  p_2^\star \Omega^{1}_{X/\ZZ_p}(D) \otimes p_2^\star \omega^{-1} \ar[u]^{\mathrm{tr}_{p_2}}  &&   p_1^\star \Omega^{1}_{X/\ZZ_p}(D) \otimes p_1^\star \omega^{-1}   \ar[u]_{\mathrm{tr}_{p_1}} \\
  p_2^\star \omega \ar[u]^{p_2^\star KS} \ar[rr]^{ \pi_1 } &&  p_1^\star \omega \ar[u]_{p_1^\star KS}}
\]
  or equivalently 
\[
  \xymatrix{  & \omega_{X_0(p)/\ZZ_p}(D_0(p)) & \\
  p_2^\star \Omega^{1}_{X/\ZZ_p}(D)  \ar[ur]^{\mathrm{tr}_{p_2}} & &  p_2^\star \Omega^{1}_{X/\ZZ_p}(D) \ar[ul]_{\mathrm{tr}_{p_1}} \\
  p_2^\star \omega^2 \ar[u]^{p_2^\star KS} \ar[rr]^{ \pi_1 \otimes (\pi_1^\vee)^{-1}} &&  p_1^\star \omega^2 .\ar[u]_{p_1^\star KS}}
\]
It remains to observe that $\pi_1^\vee \pi_1 = p$. 
\end{proof}

\begin{lem} $\DD(T_p^{\naive}) = p^{k-1} T^{t,\naive}_p$. 
\end{lem}

\begin{proof}The dual of the map $p_2^\star \omega^k \dashrightarrow p_1^\star \omega^k \rightarrow p_1^! \omega^k$ is 
\begin{equation*}
\begin{tikzcd}
p_1^\star ( \omega^{-k} \otimes \omega_{X/\ZZ_p}) \ar[r] & p_1^! (\omega^{-k} \otimes \omega_{X/\ZZ_p}) \ar[r,dashed] & p_2^! (\omega^{-k} \otimes \omega_{X/\ZZ_p}).
\end{tikzcd}
\end{equation*}
The first map $p_1^\star  \omega^{-k} \otimes p_1^\star \omega_{X/\ZZ_p} \rightarrow p_1^\star \omega^{-k} \otimes \omega_{X_0(p)/\ZZ_p}$ is just $1 \otimes \mathrm{tr}_{p_1}$. 

The second map $p_1^! (\omega^{-k} \otimes \omega_{X/\ZZ_p}) \dashrightarrow p_2^! (\omega^{-k} \otimes \omega_{X/\ZZ_p})$ is also 
\[
\begin{tikzcd}
\pi_{-k}^{-1} \otimes 1:p_1^\star \omega^{-k} \otimes \omega_{X_0(p)/\ZZ_p} \ar[r,dashed] & p_2^\star  \omega^{-k} \otimes \omega_{X_0(p)/\ZZ_p}.
\end{tikzcd}
\]

On the other hand, using the Kodaira--Spencer isomorphism we have
\[\omega_{X_0(p)/\ZZ_p} = p_1^! \oscr_{X} \otimes p_1^\star \omega_{X/\ZZ_p} = p_1^! \oscr_{X} \otimes p_1^\star  \omega^2(-D)\]
and
\[\omega_{X_0(p)/\ZZ_p} = p_2^! \oscr_{X} \otimes p_2^\star \omega_{X/\ZZ_p} = p_2^! \oscr_{X} \otimes p_2^\star  \omega^2(-D).\]
The identity  map: 
$ p_1^! \oscr_{X} \otimes p_1^\star  \omega^2(-D) \rightarrow  p_2^! \oscr_{X} \otimes p_2^\star  \omega^2(-D)$ decomposes into:
$$ \mathrm{tr}_{p_2} \mathrm{tr}_{p_1}^{-1} \otimes (\pi^\vee_1 \otimes (\pi_{1})^{-1})$$ according to Lemma~\ref{lem-KS}. 
We get that $\DD(T_p^{\naive}) :  p_1^\star (\omega^{-k+2}(-D) )  \dashrightarrow p_2^! (\omega^{-k+2}(-D))$ is
\[\mathrm{tr}_{p_2} \mathrm{tr}_{p_1}^{-1} \circ \mathrm{tr}_{p_1}\otimes (\pi^\vee_1 \otimes (\pi_{1})^{-1})  \otimes (\pi_{-k})^{-1}.\] 
It remains to observe that $\pi_{1-k}\pi^{\vee}_{1-k} = p^{1-k}$. 
\end{proof}

\begin{proof}[Proof of Proposition \ref{prop-T_p-selfdual}] This follows from the identity: $-\inf \{ 1,k\}  + k-1 = - \inf \{ 1, 2-k\}$.
\end{proof}

\section{Higher Hida theory}
\subsection{The mod $p$ theory}
We write $X_1\to\Spec\F_p$ and $X_0(p)_1\to\Spec\F_p$ for the special fibers of $X$ and $X_0(p)$.  We write $X_1^{\ord}\subseteq X_1$ for the ordinary locus.

We recall that $X_0(p)_1 = X_0(p)_1^{F} \cup X_0(p)_1^{V}$ is the union of the Frobenius and Verschiebung correspondences.  We let $p_i^{F}$ and $p_i^{V}$ be the restrictions of the projections $p_i$ to these components. The projection
\[p^V_2 : X_0(p)_1^{V} \longrightarrow X_1\]
is an isomorphism (and $X_0(p)_1^{V}$ parametrizes the Verschiebung isogeny $(p^V_1)^\star E \simeq (p^V_2)^\star E^{(p)}  \rightarrow (p^V_2)^\star E$). The projection
\[p^F_1 : X_0(p)_1^{F} \longrightarrow X_1\]  is an isomorphism (and $X_0(p)_1^{F}$ parametrizes the Frobenius isogeny $(p^F_1)^\star E \rightarrow (p^F_2)^\star E \simeq (p^F_1)^\star E^{(p)}$).  We denote by $i^F$ and $i^V $ the inclusions
\[X_0(p)_1^{F}\hookrightarrow X_0(p)_1\quad\text{and}\quad X_0(p)_1^{V} \hookrightarrow X_0(p)_1.\] 

\begin{lem}\label{lem-congruence}  If $k \geq 2$, we have a factorization on $X_1$:
\[
\xymatrix{ p_2^\star \omega^k \ar[r]^{T_p} \ar[d]& p_1^! \omega^k \\
i^V_\star  (p^V_2)^\star \omega^k \ar[r] & i^V_\star (p^V_1)^! \omega^k .\ar[u]}
\]

If $k \leq 0$, we have a factorization on $X_1$: 
\[
\xymatrix{ p_2^\star \omega^k \ar[r]^{T_p}\ar[d] & p_1^! \omega^k \\
i^F_\star  (p^F_2)^\star \omega^k \ar[r] & i^F_\star (p^F_1)^! \omega^k .\ar[u]}
\]
\end{lem}

\begin{proof} By Proposition \ref{prop-restriction-simple}, this amounts to checking that the cohomological correspondence $T_p : p_2^\star \omega^k \rightarrow p_1^! \omega^k$ vanishes at any generic point of multiplicative type in $X_0(p)_1$ if $k \geq 2$, and at any generic point of \'etale type in $X_0(p)_1$ if $k \leq 0$. This follows from the normalization of the correspondence as explained in the proof of Proposition \ref{prop-T_p-constructed}. 
\end{proof}

\begin{rem}We can informally rephrase this lemma by saying that we have congruences: $T_p = U_p\mod p$ if $k \geq 2$ and $T_p = \mathrm{Frob}\mod p$ if $k \leq 0$, see Remark \ref{rem-$q$-expansion}. 
\end{rem}

\begin{prop}\label{prop-Tp-local} For all $k \geq 2$ and $n\in\ZZ$, the cohomological correspondence $T_p$ induces a map:
$$ p_2^\star (\omega^{k} ( (np + k-2) SS ) ) \longrightarrow p_1^! (\omega^{k}( n SS)).$$
For all $k \leq 0$ and $n\in\ZZ$, the cohomological correspondence $T_p$ induces a map:
$$ p_2^\star (\omega^{k} ( -n SS ) ) \longrightarrow p_1^! (\omega^{k}( (-np +k) SS)).$$
\end{prop}

\begin{proof} We first prove the claim when $k\geq 2$. The cohomological correspondence is supported on $X_0(p)_1^V$ by Lemma~\ref{lem-congruence}. The map $p^V_1$ is totally ramified of degree $p$ and the map $p_2^V$ is an isomorphism. It follows that  we have an equality of divisors $(p^V_1)^\star(SS) = p (p^V_2)^\star(SS)$.  We deduce that the map $(p^V_2)^\star \omega^2 \rightarrow (p^V_1)^! \omega^2$ induces a morphism $(p^V_2)^\star (\omega^2(npSS)) \rightarrow (p^V_1)^!(\omega^2(nSS))$. 

This proves the claim for $k=2$. For $k \geq 3$, we remark that the cohomological correspondence 
$(p^V_2)^\star \omega^k \rightarrow (p^V_1)^! \omega^k$ is the tensor product of the map 
$(p^V_2)^\star \omega^2 \rightarrow (p^V_1)^! \omega^2$ and the map $(p^V_{2})^\star \omega^{k-2} \rightarrow (p^V_{1})^\star \omega^{k-2}$.   But $(p^V_{1})^\star \omega_E \simeq ((p^V_{2})^\star \omega_E)^{p}$ and the differential of the isogeny $(p^V_{2})^\star E \rightarrow (p^V_{1})^\star E$ identifies with the  Hasse invariant and induces an isomorphism: $(p^V_{2})^\star (\omega_E(SS)) \rightarrow (p^V_{1})^\star \omega_E$. We deduce that there is a map
$p_2^\star (\omega^{k} ( (np + k-2) SS ) ) \rightarrow p_1^! (\omega^{k}( n SS))$. 

We now prove the claim when $k\leq 0$. The cohomological correspondence is supported on $X_0(p)_1^F$ by Lemma~\ref{lem-congruence}. The map $p^F_2$ is totally ramified of degree $p$ and the map $p_1^F$ is an isomorphism. It follows that  we have an equality of divisors $(p^F_2)^\star(SS) = p (p^F_1)^\star(SS)$.  We deduce that the map $(p^F_2)^\star \oscr_X \rightarrow (p^F_1)^! \oscr_X$ induces a morphism $(p^F_2)^\star (\oscr_X(-nSS)) \rightarrow (p^F_1)^! (\oscr_X(-npSS))$. 

This proves the claim for $k=0$. For $k \leq -1$, we remark that the cohomological correspondence 
$(p^F_2)^\star \omega^k \rightarrow (p^F_1)^! \omega^k$ is the tensor product of the map 
$(p^F_2)^\star \oscr_X \rightarrow (p^F_1)^! \oscr_X$ (the cohomological correspondence for $k=0$) and a map $(p^F_{2})^\star \omega^{k} \rightarrow (p^F_{1})^\star \omega^{k}$ that we now describe.   Informally, this map is deduced from the differential of the isogeny $(p_1^F)^\star E \rightarrow (p_2^F)^\star E$, after normalizing by a factor $p^{-1}$ (one can make sense of this over the formal scheme ordinary locus). Equivalently, it is deduced from the differential of the isogeny of the dual map $(p_2^F)^\star E \rightarrow  (p_1^F)^\star E$ (the Verschiebung map). 

We observe that  $(p^F_{2})^\star \omega_E \simeq ((p^F_{1})^\star \omega_E)^{p}$ and there is a natural isomorphism:
\[(p^F_{2})^\star \omega_E \xrightarrow{(p^F_{1})^\star Ha^{-1}} (p^F_{1})^\star (\omega_{E}(SS)),\]
and therefore, for all $k \leq 0$, an isomorphism:
\[(p^F_{2})^\star \omega^k \xrightarrow{(p^F_{1})^\star Ha^{k}} (p^F_{1})^\star (\omega^{k}(-kSS))\]
which factors the  map
\[(p^F_{2})^\star \omega^k \xrightarrow{(p^F_{1})^\star Ha^{k}} (p^F_{1})^\star \omega^{k}.\]
 We deduce that there is a map: 
$p_2^\star (\omega^{k} ( -nSS )) \longrightarrow p_1^! (\omega^{k}( (-np + k) SS))$.
\end{proof}

\begin{coro}
\leavevmode
\begin{enumerate}
\item The $T_p$ operator acts on $\mathrm{R}\Gamma(X_1, \omega^k(nSS))$ for all $n \geq 0$ and $k\geq 2$, and the maps $\mathrm{R}\Gamma(X_1, \omega^k(nSS)) \rightarrow \mathrm{R}\Gamma(X_1, \omega^k(n'SS))$ are equivariant for $0 \leq n \leq n'$.

 \item We have commutative diagrams for all $n \geq 0$  and $k \geq 2$
\[
\xymatrix{ \mathrm{R}\Gamma(X_1, \omega^k((np + k-2)SS)) \ar[r]^{T_p}  \ar[dr] & \mathrm{R}\Gamma(X_1, \omega^k((np + k-2)SS)) \\
\mathrm{R}\Gamma(X_1, \omega^k(nSS)) \ar[r]^{T_p} \ar[u] & \mathrm{R}\Gamma(X_1, \omega^k(nSS))\ar[u]}
\]
where the diagonal arrow is the map of Proposition~\ref{prop-Tp-local}.
\item The $T_p$ operator acts on $\mathrm{R}\Gamma(X_1, \omega^k(-nSS))$ for all $n \geq 0$ and $k\leq 0$, and the maps $\mathrm{R}\Gamma(X_1, \omega^k(-n'SS)) \rightarrow \mathrm{R}\Gamma(X_1, \omega^k(-nSS))$ are equivariant for $0\leq n\leq n'$.

\item  We have commutative diagrams for all $n \geq 0$  and $k \leq 0$
\[
\xymatrix{ \mathrm{R}\Gamma(X_1, \omega^k(-nSS)) \ar[r]^{T_p}  \ar[dr] & \mathrm{R}\Gamma(X_1, \omega^k(-nSS)) \\
\mathrm{R}\Gamma(X_1, \omega^k((-np+k)SS)) \ar[r]^{T_p} \ar[u]  & \mathrm{R}\Gamma(X_1, \omega^k((-np+k)SS)) \ar[u] }
\]
where the diagonal arrow is the map of Proposition~\ref{prop-Tp-local}.
\end{enumerate}
\end{coro}

For any $k$, we define as usual   $\HH^i_c(X^{\ord}_1, \omega^k) = \lim_n \HH^i(X_1, \omega^{k}(-nSS))$ following \cite{MR0297775}. This is a profinite $\F_p$-vector space.   We also recall that $\HH^i(X^{\ord}_1, \omega^k) = \colim_n \HH^i(X_1, \omega^{k}(nSS))$.

\begin{coro}\label{coro-mod-p-control}
\leavevmode
\begin{enumerate}
\item If $k \geq 2$, $T_p$ is locally finite on $\HH^i(X^{\ord}_1, \omega^k)$.
\item If $k \leq 0$, $T_p$ is locally finite on $\HH^i_c(X^{\ord}_1, \omega^k)$.
\item If $k \geq 3$, we have $e(T_p) \HH^i(X^{\ord}_1, \omega^k) = e(T_p) \HH^i(X_1, \omega^k)$.
\item If $k = 2$, we have $e(T_p) \HH^i(X^{\ord}_1, \omega^2) = e(T_p) \HH^i(X_1, \omega^2(SS))$.
\item If $k \leq -1$, we have $e(T_p) \HH^i_c(X^{\ord}_1, \omega^k) = e(T_p) \HH^i(X_1, \omega^k)$.
\item If $k  =  0$, we have $e(T_p) \HH^i_c(X^{\ord}_1, \oscr_X) = e(T_p) \HH^i(X_1, \oscr_X(-SS))$.
\end{enumerate}
\end{coro}

\begin{coro} 
\leavevmode 
\begin{enumerate}
\item  If $k \leq -1$, $e(T_p) \mathrm{R}\Gamma (X_1, \omega^k)$ is concentrated in degree $1$,
\item If $k \geq 3$, $e(T_p) \mathrm{R}\Gamma (X_1, \omega^k)$ is concentrated in degree $0$.
\end{enumerate}
\end{coro}

\begin{proof} This follows from  $\HH^1(X^{\ord}_1, \omega^k) =0$  and $ \HH^0_c(X^{\ord}_1, \omega^k) = 0$ because $X^{\ord}_1$ is affine.
\end{proof} 

\begin{rem} Of course, we even have $\HH^1(X_1,\omega^k)=0$ for $k \geq 2$ and $\HH^0(X_1,\omega^k)=0$ for $k<0$ by the Riemann--Roch theorem and the Kodaira--Spencer isomorphism, but the given proof is independent of this, and more suitable for generalizations in higher dimension.
\end{rem}

\subsection{The $p$-adic theory}
Let $\mathfrak{X}$ be the $p$-adic completion of $X$ and let $X_n \rightarrow \Spec\ZZ/p^n\ZZ$ be the scheme obtained by reduction modulo $p^n$.  Let $\mathfrak{X}^{\ord}$ and $X_n^{\ord}$ denote the ordinary loci.

\subsubsection{The Igusa tower} For all $n\geq 1$, we have over $\mathfrak{X}^{\ord}$ a canonical multiplicative subgroup $H_n^{\can}\subset E[p^n]$, which is \'etale locally isomorphic to $\mu_{p^n}$.  We can consider the \'etale Cartier dual $(H_n^{\can})^D=\mathrm{Hom}(H_n^{\can},\mu_{p^n})$ (away from the boundary, the canonical principal polarization on $E$ defines an isomorphism between $(H_n^{\can})^D$ and the \'etale quotient $E[p^n]^{\et}=E[p^n]/H_n^{\can}$).  We can consider the Hodge-Tate map:
$$\mathrm{HT}:(H_n^{\can})^D\longrightarrow\omega_E/p^n$$
which sends a local section $s:H_n^{\can}\to\mu_{p^n}$ of $(H_n^{\can})^D$ to $s^\star\frac{dx}{x}\in\omega_{H_n^{\can}}=\omega_{E}/p^n$, where $\frac{dx}{x}$ is the canonical generator of the co-lie algebra of $\mu_{p^n}=\Spec\ZZ[x]/(x^{p^n}-1)$.  If $s$ is a local generator of $(H_n^{\can})^D$, then $\mathrm{HT}(s)$ is a local generator of $\omega_E/p^n$, and hence the linearized Hodge-Tate map
$$\mathrm{HT}\otimes 1:(H_n^{\can})^D\otimes_{\ZZ/p^n\ZZ}\oscr_{X_n^{\ord}}\longrightarrow \omega_E/p^n$$
is an isomorphism.

Passing to the limit over $n$ we obtain the Hodge-Tate map
$$\mathrm{HT}:T_p\left((H^{\can})^D\right)\longrightarrow \omega_E$$
where $T_p\left((H^{\can})^D\right)$ is the pro-\'etale sheaf $\lim_n (H_n^{\can})^D$ (pro-\'etale locally $\ZZ_p$), which induces an isomorphism
$$\mathrm{HT}\otimes 1:T_p\left((H^{\can})^D\right)\otimes_{\ZZ_p}\oscr_{\mathfrak{X}^{\ord}}\longrightarrow\omega_E.$$
This defines a $\ZZ_p^\times$ reduction of the principal $\mathbb{G}_m$-torsor $\omega_E$ over $\mathfrak{X}^{\ord}$ (in the pro-\'etale topology). 

We can form $\pi : \mathfrak{IG} = \mathrm{Isom} ( \ZZ_p, T_p\left((H^{\can})^D)\right) \rightarrow \mathfrak{X}^{\ord}$, the Igusa tower.  This is a $p$-adic formal scheme, and a profinite \'etale cover of $\mathfrak{X}^{\ord}$. 

 Let $\Lambda = \ZZ_p[[\ZZ_p^\times]]$ be the Iwasawa algebra, and let  $\kappa^{\un} : \ZZ_p^\times \rightarrow \Lambda^\times$ be the universal character. For any $k \in \ZZ$, we have an algebraic character $\ZZ_p^\times \rightarrow \ZZ_p^\times$, given by $x \mapsto x^k$ and we denote by $k : \Lambda \rightarrow \ZZ_p$ the corresponding algebra morphism. 
We let $\omega^{\kappa^{\un}} = \left(\oscr_{\mathfrak{IG}} \hat{\otimes} \Lambda\right)^{\ZZ_p^\times}$ where the invariants are taken for the diagonal action. 

\begin{lem} The sheaf $\omega^{\kappa^{\un}}$  is an invertible  sheaf of $\oscr_{\mathfrak{X}^{\ord}}\hat{\otimes}\Lambda$-modules  and for any $k \in \ZZ$, there is a canonical isomorphism of invertible sheaves over $\mathfrak{X}^{\ord}$: 
$$ \omega^k \rightarrow  \omega^{\kappa^{\un}} \otimes_{\Lambda, k} \ZZ_p. $$
 \end{lem}

\subsubsection{Cohomology of the ordinary locus}\label{sec-compact-def} We may now consider the following $\Lambda$-modules:
\[\HH^0(\mathfrak{X}^{\ord}, \omega^{\kappa^{\un}})\quad\text{and}\quad\HH^1_c( \mathfrak{X}^{\ord}, \omega^{\kappa^{\un}}).\]
Let us give  the definition of the second module.  Let $\mathfrak{m}_\Lambda$ be the kernel of the reduction map $\Lambda \rightarrow \F_p[ \F_p^\times]$. 
 We first define $\HH^1_c( \mathfrak{X}^{\ord}, \omega^{\kappa^{\un}}/(\mathfrak{m}_\Lambda)^n) = \HH^1_c( X_n^{\ord}, \omega^{\kappa^{\un}}/(\mathfrak{m}_\Lambda)^n)$ as follows:   we can take any extension of $ \omega^{\kappa^{\un}}/(\mathfrak{m}_\Lambda)^n$ to a coherent sheaf $\mathscr{F}$ of $\oscr_{X_n}\otimes\Lambda/(\mathfrak{m}_\Lambda)^n$-modules (this means that if  $j : X_n^{\ord} \rightarrow X_n$ is the inclusion, then $j^\star \mathscr{F} =  \omega^{\kappa^{\un}}/(\mathfrak{m}_\Lambda)^n$)  and we let $\HH^1_c( X_n^{\ord}, \omega^{\kappa^{\un}}/(\mathfrak{m}_\Lambda)^n) = \lim_l \HH^1(X_n, \mathscr{I}^l\mathscr{F} )$ for $\mathscr{I}$ any sheaf of ideals on $X_n$ whose vanishing locus is the complement of $X_n^{\ord}$.   We remark that this is a profinite $\Lambda/(\mathfrak{m}_\Lambda)^n$-module.  The profinite module $\HH^1_c( X_n^{\ord}, \omega^{\kappa^{\un}}/(\mathfrak{m}_\Lambda)^n) $ is well-defined (it does not depend on the choice of $\mathscr{I}$ or  $\mathscr{F}$) following \cite[\S 2]{MR0297775}. We then define  $\HH^1_c( \mathfrak{X}^{\ord}, \omega^{\kappa^{\un}}) = \lim_n \HH^1_c( \mathfrak{X}^{\ord}, \omega^{\kappa^{\un}}/(\mathfrak{m}_\Lambda)^n)$.
 
We also define, for $k\in\ZZ$, $\HH^1_c(X_n^{\ord},\omega^k)=\lim_l\HH^1(X_n,\mathscr{I}^l\omega^k)$, $\HH^1_c(\mathfrak{X}^{\ord},\omega^k)=\lim_n\HH^1_c(X_n^{\ord},\omega^k)$.  We remark that for the same reason as above $\HH^1_c(X_n^{\ord},\omega^k)$ can also be computed as $\lim_l\HH^1(X_n,\mathscr{I}^l\mathscr{F})$ for any coherent sheaf $\mathscr{F}$ on $X_n$ extending $\omega^k$ on $X_n^{\ord}$.  We also note that there are natural ``corestriction'' maps $\HH^1_c(X_n^{\ord},\omega^k)\to \HH^1(X_n^{\ord},\omega^k)$ and $\HH^1_c(\mathfrak{X}^{\ord},\omega^k)\to\HH^1(X,\omega^k)$.

\subsubsection{$U_p$ and Frobenius} There is a lift of Frobenius $ F : \mathfrak{X}^{\ord} \rightarrow \mathfrak{X}^{\ord}$ which is given by $E \mapsto E/H_1^{\can}$, and the $\Gamma_1(N)$ level structure on $E/H_1^{\can}$ induced by the isogeny $E\to E/H_1^{\can}$.  This map  extends to the Igusa tower as a map: $F : \mathfrak{IG}\rightarrow \mathfrak{IG}$, which is given by 
\[\left(E, \psi : \ZZ_p  \simeq T_p\left((H^{\can})^D\right) \right) \longmapsto \left(E/H_1^{\can}, \psi' : \ZZ_p \simeq T_p\left((H^{\can}/H_1^{\can})^D\right) \right)\]
where $\psi'$ is defined by  $\ZZ_p \xrightarrow{p\psi} pT_p\left((H^{\can})^D\right) \xleftarrow{\sim}  T_p\left((H^{\can}/H_1^{\can})^D\right)$.

We also consider a variant $F'$ which is defined in the same way as $F$ except that we give $E/H_1^{\can}$ the $\Gamma_1(N)$ level structure via the dual isogeny $E/H_1^{\can}\to E$, or in other words $F'=\langle p\rangle^{-1} \circ F$.  The inverse of $F'$ viewed as a correspondence is a lift over the ordinary locus of the Verschiebung correspondence on the special fiber.

We consider two maps, the natural pull-back map on functions  $F^\star \oscr_{\mathfrak{IG}} \rightarrow \oscr_{\mathfrak{IG}}$ and the trace map 
$ \mathrm{tr}_{F'}  : F'_\star \oscr_{\mathfrak{IG}} \rightarrow \oscr_{\mathfrak{IG}}$. 
\begin{lem} We have $\mathrm{tr}_{F'}(F'_\star \oscr_{\mathfrak{IG}} ) \subseteq p \oscr_{\mathfrak{IG}}$.
\end{lem}
\begin{proof} We have $\mathrm{tr}_{F'} ( F'_\star \oscr_{\mathfrak{X}^{\ord}}) \subseteq p \oscr_{\mathfrak{X}^{\ord}}$. Since  $F' \times \pi  : \mathfrak{IG} \tilde{\rightarrow} \mathfrak{IG} \times_{\pi, \mathfrak{X}^{\ord}, F'} \mathfrak{X}^{\ord}$, we deduce  that the inclusion $\mathrm{tr}_{F'}(F'_\star \oscr_{\mathfrak{IG}} ) \subseteq p \oscr_{\mathfrak{IG}}$ holds.
\end{proof}

\begin{coro}\label{coro-constructionFU_p}
 There are  two maps  $F : F^\star \omega^{\kappa^{\un}} \rightarrow \omega^{\kappa^{\un}}$ and $U_p :  F'_\star \omega^{\kappa^{\un}} \rightarrow \omega^{\kappa^{\un}}$. 
\end{coro}
\begin{proof} The maps $F^\star \oscr_{\mathfrak{IG}} \rightarrow \oscr_{\mathfrak{IG}}$ and 
$ \frac{1}{p} \mathrm{tr}_{F'}  : F'_\star \oscr_{\mathfrak{IG}} \rightarrow \oscr_{\mathfrak{IG}}$ are $\ZZ_p^\times$-equivariant. We can tensor  with $\Lambda$ and take the invariants. 
\end{proof} 

\subsubsection{$U_p$, Frobenius, and $T_p$}

We now describe the specializations of these maps at classical weight $k \in \ZZ$ and compare them with the cohomological correspondence $T_p$.  Over $\mathfrak{X}^{\ord}$ we have the canonical multiplicative isogeny $\pi_{\can}:E\to E/H_1^{\can}$ as well as its dual \'{e}tale isogeny $\pi_{\can}^\vee:E/H_1^{\can}\to E$.  We consider the induced maps on differentials $\pi_{\can}^\star:\omega_{E/H_1^{\can}}\to\omega_E$ and $(\pi_{\can}^\vee)^\star:\omega_E\to\omega_{E/H_1^{\can}}$.  The latter is an isomorphism as $\pi_{\can}^\vee$ is \'{e}tale.  Then as $(\pi_{\can}^\vee)^\star\pi_{\can}^\star=p\mathrm{Id}_{\omega_{E/H_1^{\can}}}$, we can form $p^{-1}\pi_{\can}^\star=((\pi_{\can}^\vee)^\star)^{-1}:\omega_{E/H_1^{\can}}\to\omega_E$.  We can also identify $\omega_{E/H_1^{\can}}$ with $F^\star\omega_E$ or $(F')^\star\omega_E$.

For $k\in\mathbb{Z}$, specializing the maps constructed in Corollary~\ref{coro-constructionFU_p} gives maps maps $F : F^\star \omega^{k} \rightarrow \omega^{k}$ and $U_p:F'_\star\omega^k\to\omega^k$.

\begin{lem}\label{lem-comparison-frob}    $F = p^{-k} (\pi_{\can}^\star)^{ k} :  F^\star \omega^{k} \rightarrow \omega^{k}$.
\end{lem} 

\begin{proof} The \'etale isogeny $\pi_{\can}^\vee$ induces an isomorphism $H^{\can}/H_1^{\can}\to H^{\can}$ and it follows directly from the definition of the Hodge-Tate map that there is a commutative digram
\[
\begin{tikzcd}
T_p\left((H^{\can})^D\right) \ar[r] \ar[d] &  T_p\left((H^{\can}/H_1^{\can})^D\right) \ar[d]\\
\omega_{E} \ar[r]^{(\pi_{\can}^\vee)^{\star}} & \omega_{E/H_1^{\can}}
\end{tikzcd}
\]
from which it follows easily that $F^\star \omega^k \rightarrow \omega^k$ is given by $((\pi_{\can}^\vee)^{\star})^{- k} =  p^{-k} (\pi_{\can}^\star)^{ k}$.
\end{proof}

\medskip

We can construct a map:
\[F'_\star \omega^k  \xrightarrow{((\pi_{\can}^\vee)^\star)^k}  F'_\star (F')^\star \omega^k \xrightarrow{~\frac{1}{p} \mathrm{tr}_{F'}} \omega^k.\]
\begin{lem}\label{lem-comparison-Up} The above map coincides with the map $U_p : F'_\star \omega^k \rightarrow \omega^k$.
\end{lem}
\begin{proof} Similar to Lemma~\ref{lem-comparison-frob} and left to the reader.
\end{proof}

\medskip

In Section \ref{section-T_p} we constructed a cohomological correspondence $T_p$ over $X_0(p)$:
$$T_p : p_2^\star \omega^k \longrightarrow p_1^! \omega^k.$$ 
We can consider the completion $\mathfrak{X}_0(p)$ and its ordinary part $\mathfrak{X}_0(p)^{\ord}$ which is the disjoint union of two types of irreducible components: $$ \mathfrak{X}_0(p)^{\ord}= \mathfrak{X}_0(p)^{\ord,\, F} \coprod \mathfrak{X}_0(p)^{\ord,V},$$
where on the first components the universal isogeny is not \'etale, and where it is \'etale on the other components. 

On $\mathfrak{X}_0(p)^{\ord,F}$, the map $p_1$ is an isomorphism and the map $p_2$ identifies with the Frobenius map $F$. On $\mathfrak{X}_0(p)^{\ord,V}$, the map $p_2$ is an isomorphism and the map $p_1$ identifies with $F'$.
We therefore can think of $U_p$ and $F$ as cohomological correspondences
$p_2^\star \omega^k \rightarrow p_1^! \omega^k$ supported respectively  on $ \mathfrak{X}_0(p)^{\ord,V}$ and $\mathfrak{X}_0(p)^{\ord,F}.$

One can also restrict $T_p$ to a cohomological correspondence over $\mathfrak{X}_0(p)^{\ord}$ and project it on the components $\mathfrak{X}_0(p)^{\ord,F}$ and $\mathfrak{X}_0(p)^{\ord,V}$. We denote by $T_p^F$ and $T_p^V$ the two projections of the correspondence $T_p$. 

\begin{lem}\label{lem-T_p-U_p-Frob}
\leavevmode
\begin{enumerate}
\item We have $T_p^F = p^{\sup\{0, k-1\}} F$.
\item We have  $T_p^V = p^{\sup\{0, 1-k\}} U_p$.
\item If $k \geq 1$, we have $T_p = U_p + p^{k-1} F$,
\item If $k \leq 1$, we have $T_p = F + p^{1-k} U_p$. 
\end{enumerate}
\end{lem}
\begin{proof} Parts (3) and (4) follow immediately from parts (1) and (2) (compare also with remark \ref{rem-$q$-expansion}), while parts (1) and (2) are simply a matter of bookkeeping using Lemmas~\ref{lem-comparison-frob} and~\ref{lem-comparison-Up} and the definition of $T_p$.

For the benefit of the reader we spell out the details for (1).  According to the definition of $T_p$ and in particular its normalization (recall the proof of Proposition~\ref{prop-T_p-constructed}), $p^{-\sup\{0,k-1\}}T_p^F$ exists and is given on $\mathfrak{X}_0(p)^{\ord,F}$ by the composition of $p^{-k}\pi_k:p_2^\star\omega^k\to p_1^\star\omega^k$ and the trace map for $p_1^\star\omega^k\to p_1^!\omega^k$.  But in fact $p_1$ is an isomorphism on $\mathfrak{X}_0(p)^{\ord,F}$, and making this identification $p_2$ becomes $F$, while the universal isogeny over $\mathfrak{X}_0(p)^{\ord,F}$ becomes $\pi_{\can}$.  Thus we can view $p^{-\sup\{0,k-1\}}T_p^F$ as a map $F^\star\omega^k\to\omega^k$, which is $F$ by Lemma~\ref{lem-comparison-frob}
\end{proof}

\medskip

We end this discussion with duality. 

\begin{lem}\label{U_p-F-dual} We have $\DD(F) = \langle p\rangle^{-1}U_p$.
\end{lem}

\begin{proof} Compare with Proposition~\ref{prop-T_p-selfdual}.
\end{proof}

\subsubsection{Higher Hida theory}

\begin{thm}\label{thm-higher-hida}
\leavevmode
\begin{enumerate}
\item There is a locally finite action of $F$ on $\HH^1_c( \mathfrak{X}^{\ord}, \omega^{\kappa^{\un}})$ and $ e(F) \HH^1_c( \mathfrak{X}^{\ord}, \omega^{\kappa^{\un}})$ is a finite projective $\Lambda$-module. 
Moreover, 

$$ e(F) \HH^1_c( \mathfrak{X}^{\ord}, \omega^{\kappa^{\un}}) \otimes_{\Lambda, k} \ZZ_p = e(T_p)\HH^1( {X}, \omega^{k})$$ if $k \leq -1$.

\item $U_p$ is locally finite on $\HH^0( \mathfrak{X}^{\ord}, \omega^{\kappa^{\un}})$ and $e(U_p) \HH^0( \mathfrak{X}^{\ord}, \omega^{\kappa^{\un}})$ is a finite projective $\Lambda$-module. 
Moreover, 

$$ e(U_p) \HH^0( \mathfrak{X}^{\ord}, \omega^{\kappa^{\un}}) \otimes_{\Lambda, k} \ZZ_p = e(T_p)\HH^0( {X}, \omega^{k})$$ if $k \geq 3$.

\end{enumerate}
\end{thm}

\begin{proof} We will first construct a continuous action of $F$ on $\HH^1_c({X}_n^{\ord}, \omega^{\kappa^{\un}}/(\mathfrak{m}_\Lambda)^n) $, compatible  for all $n$.

We let $\mathscr{I}\subseteq\oscr_{X_n}$ be a locally principal sheaf of ideals so that the complement of the corresponding closed subscheme is $X_n^{\ord}$ (for instance the sheaf of ideals defined by a lift of a power of the Hasse invariant).  We take a coherent sheaf $\mathscr{F}$ over $X_n$ extending  $\omega^{\kappa^{\un}}/(\mathfrak{m}_\Lambda)^n$.  We may assume that $\mathscr{F}$ is $\mathscr{I}$-torsion free by replacing $\mathscr{F}$ with its quotient by the subsheaf of $\mathscr{I}$-power torsion.  Then the multiplication map $\mathscr{I}^l\otimes_{\oscr_{X_n}} \mathscr{F}\to\mathscr{I}^l\mathscr{F}$ is an isomorphism for all $l\geq 0$, and we use this to make sense of $\mathscr{I}^l\mathscr{F}$ for all $l\in\ZZ$.  Then if $j:X_n^{\ord}\to X_n$ denotes the inclusion, we have $j_\star\omega^{\kappa^{\un}}/(\mathfrak{m}_\Lambda)^n=\colim_l\mathscr{I}^{-l}\mathscr{F}$.

We write $X_0(p)_n\to\Spec\ZZ/p^n\ZZ$ for the reduction mod $p^n$ of $X_0(p)$, and $X_0(p)_n^{\ord}$ for its ordinary locus.  Then $p_1^\star\mathscr{I}$ and $p_2^\star\mathscr{I}$ can be identified with $p_1^{-1}(\mathscr{I})\oscr_{X_0(p)_n}$ and $p_2^{-1}(\mathscr{I})\oscr_{X_0(p)_n}$, and so we view them as locally principal sheaves of ideals in $\oscr_{X_0(p)_n}$.  They define the same closed subset of $X_0(p)_n$ (the complement of $X_0(p)_n^{\ord}$) and so in particular we can pick some $m$ with $p_2^\star\mathscr{I}^m\subseteq p_1^\star\mathscr{I}$.

We have $X_0(p)_n^{\ord}  = X_0(p)_n^{\ord, F} \coprod X_0(p)_n^{\ord, V}$, where  $X_0(p)_n^{\ord, F}$ is the component where the universal isogeny has connected kernel, and  $X_0(p)_n^{\ord, V}$ the component where the universal isogeny has \'etale kernel.  The graph of the map $F : X^{\ord}_n \rightarrow X_n^{\ord}$ is $X_0(p)_n^{\ord, F}$.   We can therefore  think of    $F : F^\star \omega^{\kappa^{\un}} \rightarrow \omega^{\kappa^{\un}}$  as a cohomological correspondence on $X_0(p)_n^{\ord}$: 
$$p_2^\star  \omega^{\kappa^{\un}}/(\mathfrak{m}_\Lambda)^n \longrightarrow p_1^!  \omega^{\kappa^{\un}}/(\mathfrak{m}_\Lambda)^n$$
which is given by $F$ on the component $X_0(p)_n^{\ord, F}$ and by $0$ on the component $X_0(p)_n^{\ord, V}$.  Pushing this forward from $X_n^{\ord}$ to $X_n$ we obtain a map
$$\colim_l p_2^\star(\mathscr{I}^{-l}\mathscr{F})\longrightarrow\colim_l p_1^!(\mathscr{I}^{-l}\mathscr{F}).$$
It follows that there exists some $c$ for which there is a map $p_2^\star\mathscr{F}\to p_1^!(\mathscr{I}^{-c}\mathscr{F})$.  From this and the inclusion $p_2^\star\mathscr{I}^m\to p_1^\star\mathscr{I}$ we deduce a map
$$p_2^\star(\mathscr{I}^{lm}\mathscr{F})\longrightarrow p_1^!(\mathscr{I}^{-c+l}\mathscr{F})$$
for all $l\geq 0$.  Taking cohomology we obtain a map $F:\HH^1(X_n,\mathscr{I}^{lm}\mathscr{F})\to \HH^1(X_n,\mathscr{I}^{-c+l}\mathscr{F})$.  Passing to the limit over all $l$ we obtain a continuous endomorphism $F$ of $\HH^1_c(X_n^{\ord},\omega^{\kappa^{\un}}/(\mathfrak{m}_{\Lambda})^n)$.
 
 We now need to prove that $F$ is locally finite. We first deal with $n=1$. In that case,
 \[\HH^1_c({X}_1^{\ord}, \omega^{\kappa^{\un}}/\mathfrak{m}_\Lambda) = \bigoplus_{k = -p+2}^{0} \HH^1_c({X}_1^{\ord}, \omega^{k}) \]
 (in this last formula, we can let $k$ go through any set of representatives of $\ZZ/(p-1)\ZZ$ in $\ZZ$) and this isomorphism is equivariant for the action of $F$ on the left, and $T_p$ on the right by Lemma~\ref{lem-T_p-U_p-Frob}. It follows from Corollary~\ref{coro-mod-p-control} that  $F$ is locally   finite for $n=1$, and also that $e(F) \HH^1_c({X}_1^{\ord}, \omega^{\kappa^{\un}}/\mathfrak{m}_\Lambda) $ is a finite $\F_p$-vector space. 
 We deal with the general case by induction, using the short exact sequences in cohomology ($\mathfrak{X}^{\ord}$ is affine) and Proposition~\ref{prop-pro-exact}: 
\[ \begin{tikzcd}[column sep=small]
0 \ar[r]& \HH^1_c({X}_{n+1}^{\ord}, \omega^{\kappa^{\un}} \otimes (\mathfrak{m}_\Lambda^n/ \mathfrak{m}_\Lambda^{n+1})) \ar[r]& \HH^1_c({X}_{n+1}^{\ord}, \omega^{\kappa^{\un}}/\mathfrak{m}_\Lambda^{n+1}) \ar[r] & \HH^1_c({X}_1^{\ord}, \omega^{\kappa^{\un}}/\mathfrak{m}_\Lambda^n) \ar[r] & 0.
\end{tikzcd} \]
 
 It follows that $F$ is locally finite and $e(F) \HH^1_c( \mathfrak{X}^{\ord}, \omega^{\kappa^{\un}})$ is a finite projective $\Lambda$-module, as it is a complete flat $\Lambda$-module whose reduction mod $\mathfrak{m}$ is finite.
 
 Finally, for all $k \in\ZZ$, we have an isomorphism $e(F) \HH^1_c( \mathfrak{X}^{\ord}, \omega^{\kappa^{\un}}) \otimes_{\Lambda, k} \ZZ_p = e(F) \HH^1_c( \mathfrak{X}^{\ord}, \omega^{k})$ and we can consider the composition
\[ \begin{tikzcd}
e(F) \HH^1_c( \mathfrak{X}^{\ord}, \omega^{k})\ar[r] & \HH^1_c( \mathfrak{X}^{\ord}, \omega^{k}) \ar[r] & \HH^1( X, \omega^{k})\ar[r] & e(T_p) \HH^1( X, \omega^{k})
\end{tikzcd}\]
where the maps are inclusion, corestriction, and projection.  When $k\leq -1$ this is a map of finite free $\ZZ_p$-modules, which is an isomorphism modulo $p$ by Corollary~\ref{coro-mod-p-control}. Therefore this map is an isomorphism.

 The proof of the second point of the theorem follows along similar lines. 
 \end{proof} 

\subsection{Serre duality}   

Recall that we have a residue map  $\mathrm{res} : \HH^1(X, \Omega^1_{X/\ZZ_p}) \rightarrow \ZZ_p$.  Therefore, there is a natural map: $\HH^1_c(\mathfrak{X}^{\ord}, \omega^2(-D) \hat{\otimes} \Lambda) \rightarrow \Lambda$ which is obtained as the composite 
\begin{center}
\begin{tikzcd}[row sep=normal, column sep=normal]
\HH^1_c\left(\mathfrak{X}^{\ord}, \omega^2(-D) \hat{\otimes} \Lambda\right)\arrow[r]
& \HH^1\left(X, \omega^2(-D) \hat{\otimes} \Lambda\right) \ar[r] \arrow[d, phantom, ""{coordinate, name=Z}] & \HH^1\left(X, \omega^2(-D) \right)\otimes \Lambda \arrow[dl,
rounded corners,"\mathrm{KS}\otimes 1",
to path={ -- ([xshift=2ex]\tikztostart.east)
|- (Z) [near start]\tikztonodes
-| ([xshift=-2ex]\tikztotarget.west)
-- (\tikztotarget)}] \\
& \HH^1\left(X, \Omega^1_{X/\ZZ_p}\right) \otimes \Lambda \ar[r,"\mathrm{res}\otimes 1"]  &\Lambda
\end{tikzcd}
\end{center}
where the first map is the corestriction (see Section~\ref{sec-compact-def}).

Let us denote by  $\omega^{2-\kappa^{\un}}(-D) = \omega^2(-D) \otimes \underline{\mathrm{Hom}}(\omega^{\kappa^{\un}}, \Lambda \hat{\otimes} \oscr_{\mathfrak{X}^{\ord}})$. This is an invertible sheaf of  $\Lambda \hat{\otimes} \oscr_{\mathfrak{X}^{\ord}}$-modules over  $\mathfrak{X}^{\ord}$. 

\begin{rem} The following character  $\ZZ_p^\times  \rightarrow \Lambda^\times$, $t \mapsto t^2 (\kappa^{\un}(t))^{-1}$ induces an automorphism  $ d : \Lambda \rightarrow \Lambda$.  We have an isomorphism of $\oscr_{\mathfrak{X}^{\ord}} \hat{\otimes} \Lambda$-modules:  $\omega^{2-\kappa^{\un}}(-D) = \omega^{\kappa^{\un}}(-D) \otimes_{\Lambda, d} \Lambda$.
\end{rem}

We can therefore define a pairing:
$$ \langle -,- \rangle :  \HH^0\left(\mathfrak{X}^{\ord}, \omega^{\kappa^{\un}}\right) \times \HH^1_c\left(\mathfrak{X}^{\ord}, \omega^{2-\kappa^{\un}}(-D)\right) \longrightarrow \HH^1_c\left( \mathfrak{X}^{\ord}, \omega^{2}(-D) \otimes_{\ZZ_p} \Lambda\right) \longrightarrow \Lambda. $$

\begin{prop}\label{propFU_ptranspose} For any $(f, g) \in \HH^0(\mathfrak{X}^{\ord}, \omega^{\kappa^{\un}}) \times \HH^1_c(\mathfrak{X}^{\ord}, \omega^{2-\kappa^{\un}}(-D))$, we have $\langle \langle p\rangle^{-1}U_p f, g \rangle = \langle f, F g \rangle$.
\end{prop}

\begin{proof} We have a commutative diagram: 
\begin{center}
\begin{tikzcd}
\HH^0\left(\mathfrak{X}^{\ord}, \omega^{\kappa^{\un}}\right) \times \HH^1_c\left(\mathfrak{X}^{\ord}, \omega^{2-\kappa^{\un}}(-D)\right) \ar[r] \ar[d] & \Lambda \ar[d] \\
\prod_{k \in \ZZ} \HH^0\left(\mathfrak{X}^{\ord}, \omega^{k}\right) \times \HH^1_c\left(\mathfrak{X}^{\ord}, \omega^{2-k}(-D)\right) \ar[r]  & \prod_{k\in \ZZ} \ZZ_p
\end{tikzcd}
\end{center} 
where the vertical maps are injective. (The injectivity of the first vertical map follows from the fact that for any complete flat $\Lambda$-module $M$, $M\to\prod_k M\otimes_{\Lambda,k}\ZZ_p$ is injective.  Indeed, if $\mathfrak{p}_k=\ker(k:\Lambda\to \ZZ_p)$, then tensoring the injective map $\Lambda/(\mathfrak{p}_0\cdots\mathfrak{p}_k)\to\prod_{i=0}^k \Lambda/\mathfrak{p}_i$ with $M$, we see that the kernel is contained in $\bigcap_k\mathfrak{p}_0\cdots\mathfrak{p}_kM\subseteq \bigcap_n(\mathfrak{m}_\Lambda)^nM=0$ by completeness.)  It suffices therefore to prove the identity for the pairing 
\[\langle -,- \rangle_k :  \HH^0(\mathfrak{X}^{\ord}, \omega^{k}) \times \HH^1_c(\mathfrak{X}^{\ord}, \omega^{2-k}(-D)) \longrightarrow \HH^1_c( \mathfrak{X}^{\ord}, \omega^{2}(-D) ) \longrightarrow \ZZ_p \]

We work modulo $p^n$.  As in the proof of Theorem~\ref{thm-higher-hida} we let $\mathscr{I}$ be a locally principal sheaf of ideals defining a closed subscheme whose complement is $X_n^{\ord}$, and we have a cohomological correspondence over $X_0(p)_n$
$$ F  : p_2^\star  \mathscr{I}^{ml}\omega^{2-k}(-D)  \longrightarrow p_1^!  \mathscr{I}^{-c+ l}\omega^{2-k}(-D) $$
so that passing to cohomology and taking the limit over $l$ gives the action of $F$ on $\HH^1_c({X}_n^{\ord}, \omega^{2-k}(-D))$.

We can consider the dual $\DD(F):p_1^\star  \mathscr{I}^{-l+c}\omega^{k}  \rightarrow p_2^!  \mathscr{I}^{-ml}\omega^{k}$.  Passing to cohomology and taking the colimit over $l$ gives the action of $\langle p\rangle^{-1}U_p$ on $\HH^0(X_n^{\ord},\omega^k)$ by Lemma~\ref{U_p-F-dual}.
\end{proof} 

This pairing hence restricts to a pairing: 
\[\langle -,- \rangle :  e(U_p)\HH^0(\mathfrak{X}^{\ord}, \omega^{\kappa^{\un}}) \times e(F)\HH^1_c(\mathfrak{X}^{\ord}, \omega^{2-\kappa^{\un}}(-D)) \longrightarrow  \Lambda.\]

\begin{thm}
\leavevmode
\begin{enumerate} 
\item The pairing $\langle -,- \rangle$ is a perfect pairing, 
\item  For any $(f, g) \in e(U_p)\HH^0(\mathfrak{X}^{\ord}, \omega^{\kappa^{\un}}) \times e(F)\HH^1_c(\mathfrak{X}^{\ord}, \omega^{2-\kappa^{\un}}(-D))$,
\[\langle \langle p\rangle^{-1}U_p f, g \rangle = \langle f, F g \rangle,\]
\item The pairing  $\langle -,- \rangle$ is compatible with the classical pairing in the sense that for any $k \in \mathbb{Z}$, we have a commutative diagram, where the bottom pairing is the one deduced from Serre duality on $X$: 
\begin{center}
\begin{tikzcd}
 e(U_p)\HH^0\left(\mathfrak{X}^{\ord}, \omega^{k}\right)  \times  e(F) \HH^1_c\left(\mathfrak{X}^{\ord}, \omega^{2-k}(-D)\right) \ar[r] \ar[d,"j",xshift=40pt] & \ZZ_p.  \\
e(T_p)\HH^0\left(X, \omega^{k}\right) \ar[u,"i",xshift=-40pt] \times e(T_p)\HH^1\left(X, \omega^{2-k}(-D)\right) \ar[ur,start anchor=north east,end anchor=south west]  &
\end{tikzcd}
\end{center}

\end{enumerate}
\end{thm}

\begin{proof}  The second point follows from Proposition~\ref{propFU_ptranspose}. The first point will follow from the third point, since the map $i$ and $j$ are isomorphisms for integers $k \geq 3$ and the bottom pairing is perfect. 
Let us prove the last point. First, we consider the diagram  without applying projectors: 
\begin{center}
\begin{tikzcd}
\HH^0\left(\mathfrak{X}^{\ord}, \omega^{k}\right)  \times  \HH^1_c\left(\mathfrak{X}^{\ord}, \omega^{2-k}(-D)\right) \ar[r] \ar[d,"j",xshift=40pt] & \ZZ_p.  \\
\HH^0\left(X, \omega^{k}\right) \ar[u,"i",xshift=-40pt] \times \HH^1\left(X, \omega^{2-k}(-D)\right) \ar[ur,start anchor=north east,end anchor=south west]  &
\end{tikzcd}
\end{center}
which is  commutative by construction. For any $f \in \HH^0(X, \omega^{k})$ and $g \in  \HH^1_c(\mathfrak{X}^{\ord}, \omega^{2-k}(-D))$, we have $\langle i(f), g \rangle = \langle f, j(g) \rangle$. If we now assume that $f \in  e(T_p)\HH^0(X, \omega^{k})$ and $g \in e(F)  \HH^1_c(\mathfrak{X}^{\ord}, \omega^{2-k}(-D))$, we have that 
$$\langle e(U_p) i(f), g \rangle = \langle i(f), e(F) g \rangle = \langle i(f), g \rangle$$ and 
$$\langle f,  e(T_p)j(g) \rangle = \langle e(T_p) f, j( g) \rangle = \langle f, j(g) \rangle$$ and the conclusion follows. 
\end{proof}

\section{Higher Coleman theory}

\subsection{Cohomology with support in a closed subspace} We recall the notion of  cohomology of an abelian sheaf on a topological space,  with support in a closed subspace. A reference for this material is \cite[Expos\'e~I]{MR2171939}.
Let $X$ be a topological space. 
Let $i : Z \hookrightarrow X$ be a closed subspace. We denote by $C_{Z}$ and $C_{X}$ the categories of sheaves of abelian groups  on $Z$ and $X$ respectively.

We have the  pushforward functor $i_\star : C_{Z} \rightarrow C_{X}$. The functor $i_\star$  has a right adjoint $i^! : C_{X} \rightarrow C_{Z}$.  For an abelian sheaf $\mathscr{F}$ over $X$, we let $\Gamma_Z(X, \mathscr{F}) = \HH^0(X, i_\star i^! \mathscr{F})$. By definition this is the subgroup of $\HH^0(X, \mathscr{F})$ of sections whose support is included in $Z$. We let $\mathrm{R}\Gamma_Z(X, -)$ be the derived functor of $\Gamma_Z(X,-)$. 
  
Let $U= X \setminus Z$ and let $\mathscr{F}$ be an object of $C_X$. We have an exact triangle \cite[I, Corollary~2.9]{MR2171939}:
\[ \mathrm{R}\Gamma_Z(X, \mathscr{F}) \longrightarrow \mathrm{R}\Gamma(X, \mathscr{F}) \longrightarrow \mathrm{R}\Gamma(U, \mathscr{F}) \xrightarrow{+1}\]
  
Some  properties of the cohomology with support are: 
\begin{enumerate}
\item (Change of support) If $Z \subset Z'$, there is a map $\mathrm{R}\Gamma_Z(X, \mathscr{F}) \rightarrow \mathrm{R}\Gamma_{Z'}(X, \mathscr{F})$ (\emph{cf.} \cite[Expos\'e~I, Proposition~1.8]{MR2171939} ).

\item (Pull-back) If we have a cartesian diagram:
\[
\xymatrix{ Z \ar[r] \ar[d]& X \ar[d]^f \\
Z' \ar[r] & X'}
\]
and a sheaf $\mathscr{F}$ on $X'$, there is a map $\mathrm{R}\Gamma_{Z'}(X', \mathscr{F}) \rightarrow  \mathrm{R}\Gamma_{Z}(X, f^\star \mathscr{F})$,

\item (Change of ambient space) If we have $Z \subset U \subset X$ for some open $U$ of $X$, then the pull back map   $\mathrm{R}\Gamma_Z(X, \mathscr{F}) \rightarrow \mathrm{R}\Gamma_{Z}(U, \mathscr{F}) $ is a quasi-isomorphism (\emph{cf.} \cite[I, Proposition~2.2]{MR2171939}).
\end{enumerate}

We now discuss the construction of the trace map in the context of adic spaces and finite flat morphisms.

\begin{lem}\label{lem-trace-1}  Consider a commutative  diagram  of  topological spaces:

\begin{eqnarray*}
\xymatrix{ Z \ar[r] \ar[d]& X \ar[d]^f \\
Z' \ar[r] & X'}
\end{eqnarray*}
with $X$ and $X'$ adic spaces, $f$ a finite flat morphism of adic spaces, $Z'$ and $Z$ are closed subspaces of $X'$ and $X$ respectively. 
Let  $\mathscr{F}$ be a sheaf of $\oscr_{X'}$-modules. Then there is a trace map $\mathrm{R}\Gamma_{Z}(X, f^\star \mathscr{F}) \rightarrow  \mathrm{R}\Gamma_{Z'}(X', \mathscr{F})$.

\end{lem}
\begin{proof} We first recall that the category of sheaves of $\oscr_T$-modules on a ringed space $(T, \oscr_T)$ has enough injectives (\cite[Tag 01DH]{stacks-project}). It follows that it is enough to construct a functorial map $\Gamma_{Z}(X, f^\star \mathscr{F}) \rightarrow  \Gamma_{Z'}(X', \mathscr{F}) $  for  sheaves  $\mathscr{F}$ of $\oscr_{X'}$-modules. 
We have a map $\Gamma_{Z}(X, f^\star \mathscr{F}) \rightarrow  \Gamma_{f^{-1}(Z')}(X, f^\star \mathscr{F})$. Therefore, it suffices to consider the case where $Z = f^{-1}(Z')$. 
We have a trace map $\Tr : f_\star f^\star \mathscr{F} \rightarrow \mathscr{F}$.  Let us complete the above diagram into:
\begin{eqnarray*}
\xymatrix{ Z \ar[r] \ar[d]& X \ar[d]^f & U \ar[d]^g \ar[l]\\
Z' \ar[r] & X' & U' \ar[l]_{j'}}
\end{eqnarray*}
where $U' = X' \setminus Z'$ and $U = X \setminus Z$.  We have a commutative diagram:
\begin{eqnarray*}
\xymatrix{ f_\star f^\star \mathscr{F} \ar[r] \ar[d]^{\Tr} & j'_\star g_\star g^\star (j')^\star \mathscr{F} \ar[d] \\
\mathscr{F} \ar[r] & j'_\star  (j')^\star \mathscr{F} }
\end{eqnarray*}
Taking global sections and the induced map on the kernel of the two horizontal morphisms, we deduce that there is a map $\Gamma_Z(X, f^\star \mathscr{F}) \rightarrow \Gamma_{Z'}(X', \mathscr{F})$. 
\end{proof}

\subsection{The modular curve $X_0(p)$} We let $X_0(p)$ be the compactified modular curve of level $\Gamma_0(p)$ and tame level $\Gamma_1(N)$ for some prime to $p$ integer $N\geq 3$, viewed as an adic space over $\Spa( \qq_p, \ZZ_p)$. We let $H_1 \subset E[p]$ be the universal subgroup of order $p$. 

\subsubsection{Parametrization by the degree} Let $X_0(p)^{\mathrm{rk }1}$ be the subset of rank one points of $X_0(p)$.  Fargues defines a map $\deg : X_0(p)^{\mathrm{rk }1} \rightarrow [0,1]$, which sends $x  \in X_0(p)^{\mathrm{rk }1}$ to $\deg H_{1} \in  [0,1]$. We briefly recall the definition (see \cite[\S4, Definition~3]{MR2673421} for more details). The group $H_1$ extends to a finite flat group scheme $\tilde{H}_1$ over $\Spec~k(x)^+$, by taking the schematic closure in $E[p]$. There is an isomorphism $\omega_{\tilde{H}_1} = k(x)^+/fk(x)^+$ and $\deg H_{1} = v_x(f)$ for $v_x$ the valuation attached to $x$, normalized by $v(p)=1$. For any rational interval $[a,b] \subset [0,1]$, there is a unique quasi-compact open $X_0(p)_{[a,b]}$ of $X$ such that $\deg^{-1}[a,b] = X_0(p)^{\mathrm{rk} 1}_{[a,b]}$.  We let $X_0(p)_{[0,a[ \cup ]b, 1]} = X_0(p) \setminus X_0(p)_{[a,b]}$. 

\begin{rem} We remark that in this parametrization, the two extremal points $0$ (resp. $1$) correspond to ordinary semi-abelian schemes equipped with an \'etale (resp. multiplicative) subgroup $H_1$.
\end{rem}

 \subsubsection{The canonical subgroup}\label{sec-can-sub} We let $X$ be the compactified modular curve of level prime to $p$. 
 We have an Hasse invariant $\mathrm{Ha}$ and we can define the Hodge height:
 $$\mathrm{Hdg} :  X^{\mathrm{rk 1}} \rightarrow [0,1]$$ obtained by sending $x$ to $\inf \{ v_x(\widetilde{\mathrm{Ha}}), 1 \}$ for any local lift $\widetilde{\mathrm{Ha}}$ of the Hasse invariant.  For any $v \in [0,1]$, we let $X_v = \{ x \in X \mid\mathrm{Hdg}(x) \leq v\}$ (more correctly, $X_v$ is the quasi-compact open whose rank one points are those described as above). 
 
 We recall the following theorems:
 
 \begin{thm}[{\cite[Theorem~3.1]{MR0447119}}]
 If $v < \frac{p}{p+1}$,  then over $X_v$ we have a canonical subgroup $H_1^{\can} \subset E[p]$ which is locally isomorphic to $\ZZ/p\ZZ$ in the \'etale topology.
 \end{thm}

 \begin{thm}\label{thm-can}
\leavevmode 
 \begin{enumerate}
   \item For any rank one point of $X_0(p)$, we have the identity $\sum_{H \subset E[p]} \deg H = 1$. Moreover, either all the degrees are equal or there exists a canonical subgroup $H_1^{\can}$ and for all $H \neq H_1^{\can}$, $\deg H = \frac{1 - \deg H_1^{\can}}{p}$. 
  \item  If $a< \frac{1}{p+1}$, $X_0(p)_{[0,a]}$ carries a canonical subgroup $H_1^{\can} \neq H_1$ and $\deg (H_1) = \frac{ \mathrm{Hdg}}{p}$. 
  \item  If $a > \frac{1}{p+1}$, $X_0(p)_{[a,1]}$ carries a canonical subgroup $H_1^{\can} = H_1$, and $\deg(H_1) = 1- \mathrm{Hdg}$. 
 \end{enumerate}
\end{thm}
\begin{proof} See \cite[A.2]{MR2783930} and \cite[Theorem~6]{MR2919687}.
\end{proof}

\subsection{The correspondence $U_p$}\label{sect-corres-Up} We let $C$ be the correspondence over $X_0(p)$ underlying $U_p$. It parametrizes isogenies of degree $p$: $(E \rightarrow E', H_1 \stackrel{\simeq}\rightarrow H_1')$. We denote by $H = \mathrm{Ker} (E \rightarrow E')$. We have two projections $p_1((E,H_1, E', H'_1)) = (E, H_1)$, $p_2((E,H_1, E', H'_1)) = (E', H'_1)$. 

There is actually an isomorphism $X_0(p^{2}) \rightarrow C$ (where $X_0(p^2)$ parametrizes $(E,H_2 \subset E[p^2])$, with $H_2$ locally isomorphic to $\ZZ/p^2\ZZ$), mapping $(E, H_{2})$ to $(E/H_1, H_{2}/H_1, E, H_1) $   where  $H_1 = H_2[p]$, and the isogeny $E/H_1 \rightarrow E$ is dual to $E \rightarrow E/H_1$.  The inverse of this isomorphism sends $(E,H_1, E', H'_1)$ to $(E', p^{-1} H_1/H)$. 
 
  Let us denote $C_{[a,b]} = p_1^{-1}( X_0(p)_{[a,b]})$. By Theorem~\ref{thm-can}, 
  if $a > \frac{1}{p+1}$, then we have a canonical subgroup of order $p$ over $X_0(p)_{[a,1]}$, $H_1^{\can} = H_1$ and
  if $a < \frac{1}{p+1}$, we also have a canonical subgroup of order $p$ over $X_0(p)_{[0,a]}$ and $H_1 \neq H_1^{\can}$. 
 The map $p_1 : C_{[0,a]} \rightarrow X_0(p)_{[0,a]}$ has a section given by $H_1^{\can}$. Let $C^{\can}_{[0,a]}$ be the image of this section and let $C^{\et}_{[0,a]}$ be its complement, so that $C_{[0,a]} = C^{\can}_{[0,a]} \coprod C^{\et}_{[0,a]}$.

\begin{prop}\label{prop-dyn-1}
\leavevmode
  \begin{enumerate}
  \item If $a \geq \frac{1}{p+1}$, $p_2( C_{\{a\}})  = X_0(p)_{\{\frac{p-1}{p} + \frac{a}{p}\}}$. 
  \item If $a < \frac{1}{p+1}$, we have that $p_2( C^{\can}_{\{a\}}) = X_0(p)_{\{pa\}}$, and $p_2(C^{\et}_{\{a\}})  = X_0(p)_{\{1-a\}}$.
  \item If $a \in ]0,1[$, we have $\overline{p_2(C_{[a,1]})} \subseteq  X_{[a,1]}$.
  \end{enumerate}
\end{prop}
\begin{proof} We do the case by case argument to check the first two points, using Theorem~\ref{thm-can}, (1):
\begin{enumerate}
\item If $a>\frac{1}{p+1}$, we have $H_1 = H_1^{\can}$.  If $H \subset E[p]$ satisfies $H \neq H_1^{\can}$, then $\deg H = \frac{1-\deg H_1}{p}$ and $\deg E[p]/H = \frac{p-1}{p} + \frac{\deg H_1}{p}$. If $a=\frac{1}{p+1}$ there is no canonical subgroup, so $\deg H = \frac{1}{p+1}$ and the computation is the same. 

\item  On $C^{\can}_{[0,a]}$, the isogeny $E \rightarrow E/H$ is the canonical isogeny and $\deg E[p]/H = 1- \deg H$ where $\deg H = {1-p\deg H_1}$. 
On $C^{\et}_{[0,a]}$ we have $H \neq H_1^{\can}$, and $\deg H = \deg H_1$. Therefore $\deg E[p]/H = 1- \deg H_1$. 
\end{enumerate}
We deduce that if $a\in ]0,1[$, there exists $b>a$ such that ${p_2(C_{[a,1]})} \subseteq  X_{[b,1]}$, and the last point follows. 
\end{proof}

\medskip

We now give a similar analysis for the transpose of  $U_p$. It is useful to use the isomorphism $C \simeq X_0(p^2)$,  for which we have  $p_2(E, H_2) = (E,H_1)$ and $p_1(E, H_2) = (E/H_1, H_2/H_1)$.  We let $C^{[a,b]} = p_2^{-1}( X_0(p)_{[a,b]})$. 

If $\deg H_1 < \frac{p}{p+1}$, we deduce that $\deg E[p]/H_1 > \frac{1}{p+1}$, and hence $E[p]/H_1$ is the canonical subgroup. If $\deg H_1 > \frac{p}{p+1}$, we deduce that $\deg E[p]/H_1 < \frac{1}{p+1}$, and hence $E[p]/H_1$ is not the canonical subgroup, but $E/H_1$ admits a canonical subgroup. 
We denote by $C^{[a,1],\,\can}$ the component where $H_2/H_1$ is the canonical subgroup and by $C^{[a,1], \et}$ its complement. 

\begin{prop}\label{prop-dyn-2} 
\leavevmode
\begin{enumerate}
  \item If $a \leq \frac{p}{p+1}$, $p_1( C^{\{a\}})  = X_0(p)_{\{\frac{a}{p}\}}$. 
  \item If $a > \frac{p}{p+1}$, we have that $p_1( C^{\{a\},\,\can}) = X_0(p)_{\{1- p(1-a)\}}$, and $p_1(C^{\{a\},\,\et}))  = X_0(p)_{\{1-a\}}$.
  \item For any $0 < a < 1$, we have that $p_1(C^{[0,a[}) \subseteq (X_0(p)_{[0,a[})^\circ$, the interior of $X_0(p)_{[0,a[}$.
  \end{enumerate}
\end{prop}
\begin{proof} We do the case by case argument for the first two points, using Theorem~\ref{thm-can}, (1):  \begin{enumerate}
\item We have $\deg E[p]/H_1 = 1- \deg H_1$.  If $a < \frac{p}{p+1}$, this is the canonical subgroup. We deduce that $\deg H_2/H_1 = \frac{\deg H_1}{p}$. If  $a = \frac{p}{p+1}$, there is no canonical subgroup and the same formula holds. 
\item We have $\deg E[p]/H_1 = 1- \deg H_1$, and hence $E[p]/H_1$ is not the canonical subgroup. In case $(E,H_2) \in C^{[a,1],\,\can}$, we deduce that $\deg H_2/H_1 = 1- p(1- \deg H_1)$. If $(E, H_2) \in C^{[a,1],\,\et}$, we deduce that $\deg H_1/H_2= 1- \deg H_1$. 
\end{enumerate}
We deduce that if $0 < a < 1$, there exists $b<a$ such that $p_1(C^{[0,a[}) \subseteq {X_0(p)_{[0,b[}}$. 
\end{proof}

\subsection{The $U_p$-operator}

We work over $X_0(p)$.  Let $a \in ]0,1[ \cap \qq$. The cohomologies of interest are:

\begin{enumerate}
\item $\mathrm{R}\Gamma ( X_0(p), \omega^k)$,
\item $\mathrm{R} \Gamma( X_0(p)_{[a,1]}, \omega^k)$,
\item $\mathrm{R}\Gamma_{X_0(p)_{[0,a[}} ( X_0(p), \omega^k)$.
\end{enumerate}

The category of perfect  Banach complexes is the homotopy category of the category of bounded complexes of Banach spaces over $\qq_p$.  A quasi-isomorphism of perfect Banach complexes admits an inverse up to homotopy (\cite[Proposition~1.3.22]{MR1779315},  a quasi-isomorphism is always strict by the open mapping theorem).

\begin{lem}  The above  cohomologies can be canonically  represented by  objects of the category of perfect Banach complexes by using $\check{C}$ech covers. 
\end{lem}
\begin{proof} By \cite[Lemma~2.6]{MR1306024},  any finitely generated module over a complete Tate algebra carries a canonical topology and is complete. We deduce that the sections of $\omega^k$ over an affinoid open subset of $X_0(p)$ form naturally a $\qq_p$-Banach space. In case $(1)$  and $(2)$ we can take a \v{C}ech complex for a finite affinoid covering to represent the cohomology. Since any two \v{C}ech cover can be refined by a third one, and since quasi-isomorphism admit inverses up to homotopy, we deduce the independence (up to homotopy) of the \v{C}ech complex.  In case $(3)$, the cohomology fits in an exact triangle: 
$$\mathrm{R}\Gamma_{X_{[0,a[}} ( X_0(p), \omega^k) \rightarrow \mathrm{R}\Gamma ( X_0(p), \omega^k) \rightarrow  \mathrm{R}\Gamma ( X_0(p)_{[a,1]}, \omega^k) \stackrel{+1}\rightarrow$$
and is quasi-isomorphic to the cone of a continuous  map between perfect Banach complexes, hence is a also a perfect Banach complex. 

 \end{proof}

 \subsubsection{Constructing the action}
 A morphism in the category of  perfect Banach complexes is called compact if it can be represented by a morphism between complexes which is compact in each degree.

 We define a naive cohomological correspondence $U_p^{\naive}$ has follows. The two ingredients are the differential  map $p_2^\star \omega_E \rightarrow p_1^\star \omega_E$ and the trace map $(p_1)_\star \oscr_{C} \rightarrow \oscr_{X_0(p)}$. Putting all this together, we obtain
 \[ U_p^{\naive} : (p_1)_\star  p_2^\star \omega^k \longrightarrow  \omega^k.\]

\begin{prop}\label{lem-construction-U_p^{naive}} We have an action of the $U^{\naive}_p$-operator on $ \mathrm{R}\Gamma ( X_0(p), \omega^k)$,
$\mathrm{R} \Gamma( X_0(p)_{[a,1]}, \omega^k)$ and  $\mathrm{R}\Gamma_{X_0(p)_{[0,a[}} ( X_0(p), \omega^k)$ for any $0 < a < 1$. The $U^{\naive}_p$-operator is compact.  Moreover, the  $U_p^{\naive}$-operator acts equivariantly on the triangle:
\[ \mathrm{R}\Gamma_{X_0(p)_{[0,a[}} (X_0(p), \omega^k) \longrightarrow \mathrm{R}\Gamma(X_0(p), \omega^k) \longrightarrow \mathrm{R}\Gamma(X_0(p)_{[a,1]}, \omega^k)  \xrightarrow{+1}  \]
\end{prop}
\begin{proof}  The operator  $U_p^{\naive}$ is compact on $\mathrm{R}\Gamma (X_0(p), \omega^k)$ because this later complex is a perfect complex of finite dimensional $\qq_p$-vector spaces. By Proposition~\ref{prop-dyn-1}:   $$\overline{p_2 (C_{[a, 1]})} \subset X_0(p)_{[a,1]}.$$ 
It follows that we have a compact morphism $\mathrm{R}\Gamma ( X_0(p)_{[a,1]}, \omega^k) \rightarrow  \mathrm{R}\Gamma ( p_2 (C_{[a,1]}), \omega^k)$.
Therefore, we have a map:
\begin{center}
\begin{tikzcd}[row sep=normal, column sep=normal]
\mathrm{R}\Gamma \left( X_0(p)_{[a,1]}, \omega^k\right) \arrow[r]
& \mathrm{R}\Gamma \left( p_2 (C_{[a,1]}), \omega^k\right)  \ar[r,"p_2^\star"] \arrow[d, phantom, ""{coordinate, name=Z}] &  \mathrm{R}\Gamma \left( C_{[a,1]}, p_2^\star \omega^k\right) \arrow[dl,
rounded corners,
to path={ -- ([xshift=2ex]\tikztostart.east)
|- (Z) [near start]\tikztonodes
-| ([xshift=-2ex]\tikztotarget.west)
-- (\tikztotarget)}] \\
&\mathrm{R}\Gamma \left( C_{[a,1]}, p_1^\star \omega^k\right) \ar[r,"\mathrm{trace}"]  &\mathrm{R}\Gamma \left( X_0(p)_{[a,1]}, \omega^k\right).
\end{tikzcd}
\end{center}

We deduce that $U^{\naive}_p$ acts and is compact on $\mathrm{R} \Gamma( X_0(p)_{[a,1]}, \omega^k)$ because the first map
\[ \mathrm{R}\Gamma ( X_0(p)_{[a,1]}, \omega^k) \longrightarrow  \mathrm{R}\Gamma ( p_2 (C_{[a,1]}), \omega^k)\]
is.  Similarly, by Proposition~\ref{prop-dyn-2}:
\[{p_1(C^{[0, a[})} \subset \left(X_0(p)_{[0,a[}\right)^\circ.\]
The operator $U_p^{\naive}$ acts like the composite of  the following maps: 
\begin{center}
\begin{tikzcd}[row sep=normal, column sep=normal]
\mathrm{R}\Gamma_{X_{[0,a[}} \left( X_0(p), \omega^k\right) \arrow[r,"p_2^\star"]
& \mathrm{R}\Gamma_{C^{[0,a[}} \left( C, p_2^\star \omega^k\right) \ar[r] \arrow[d, phantom, ""{coordinate, name=Z}] &  \mathrm{R}\Gamma_{C^{[0,a[}} \left( C, p_1^\star \omega^k\right) \arrow[dll,
rounded corners,
to path={ -- ([xshift=2ex]\tikztostart.east)
|- (Z) [near start]\tikztonodes
-| ([xshift=-2ex]\tikztotarget.west)
-- (\tikztotarget)}] \\
\mathrm{R}\Gamma_{ p_1^{-1} p_1(C^{[0,a[})} \left( C, p_1^\star \omega^k\right) \ar[r,"\mathrm{trace}"]  &\mathrm{R}\Gamma_{p_1(C^{[0,a[})} \left( X_0(p),  \omega^k\right)\ar[r] &\mathrm{R}\Gamma_{X_0(p)_{[0,a[}} \left( X_0(p), \omega^k\right).
\end{tikzcd}
\end{center}

We claim that the map $\mathrm{R}\Gamma_{p_1(C{[0,a[})} ( X_0(p),  \omega^k)  \rightarrow \mathrm{R}\Gamma_{X_0(p)_{[0,a[}} ( X_0(p), \omega^k)$ is compact. This will follow  if we  prove that the map $\mathrm{R}\Gamma_{X_0(p)_{[0,b[}} ( X_0(p),  \omega^k)  \rightarrow \mathrm{R}\Gamma_{X_0(p)_{[0,a[}} ( X_0(p), \omega^k)$ for $0 \leq b< a \leq 1$ is compact. To see this, we observe that there is a map of exact triangles:
\begin{center}
\begin{tikzcd} \mathrm{R}\Gamma_{X_0(p)_{[0,b[}} \left( X_0(p),  \omega^k\right) \ar[r] \ar[d] & \mathrm{R}\Gamma \left( X_0(p),  \omega^k\right) \ar[r] \ar[d] &  \mathrm{R}\Gamma \left( X_0(p)_{[b,1]},  \omega^k\right) \ar[d] \ar[r,"+1"]  & {}\\
\mathrm{R}\Gamma_{X_0(p)_{[0,a[}} \left( X_0(p),  \omega^k\right) \ar[r]  & \mathrm{R}\Gamma \left( X_0(p),  \omega^k\right) \ar[r]  &  \mathrm{R}\Gamma \left( X_0(p)_{[a,1]},  \omega^k\right) \ar[r,"+1"] & {}
\end{tikzcd}
\end{center}
and since the two vertical maps on the right are compact, the first vertical map is also compact. 
 \end{proof}

For any $h \in \qq$ we can consider a direct factor $ \mathrm{R}\Gamma ( X_0(p), \omega^k)^{\leq h}$,
$\mathrm{R} \Gamma( X_0(p)_{[a,1]}, \omega^k)^{\leq h}$ and  respectively $\mathrm{R}\Gamma_{X_0(p)_{[0,a[}} ( X_0(p), \omega^k)^{\leq h}$ of $ \mathrm{R}\Gamma ( X_0(p), \omega^k)$,
$\mathrm{R} \Gamma( X_0(p)_{[a,1]}, \omega^k)$ and  respectively $\mathrm{R}\Gamma_{X_0(p)_{[0,a[}} ( X_0(p), \omega^k)$ called the slope  less than $h$ part. It is obtained by representing $U_{p}^{\naive}$ by a compact map of complexes and by applying the slope less than $h$ projector in each degree (\cite{MR144186}). The slope $\leq h$ complexes are perfect complexes of $\qq_p$-vector spaces (\textit{i.e.} they are bounded complexes with finite dimensional cohomology). The finite slope part (denote by a supscript $fs$) is the inverse limit of the $\leq h$-part for $h \rightarrow \infty$.

We also note the following corollary of the proof: 

\begin{coro} For any $0 < a < b <1$, the natural maps
$$ \mathrm{R}\Gamma(X_0(p)_{[a,1]}, \omega^k) \longrightarrow \mathrm{R}\Gamma(X_0(p)_{[b,1]}, \omega^k)$$
and $$ \mathrm{R}\Gamma_{X_0(p)_{[0,a[}}(X_0(p), \omega^k) \longrightarrow \mathrm{R}\Gamma_{X_0(p)_{[0,b[}}(X_0(p), \omega^k)$$ induce quasi-isomorphism on the finite slope part for $U_p^{\naive}$. 
\end{coro}

\subsubsection{ $U_p$ and Frobenius} \label{sec-up-frob}

It is worth spelling out the action of $U_p$ on $\mathrm{R}\Gamma_{X_0(p)_{[0,a[}} ( X_0(p), \omega^k)$. 

If $a < \frac{1}{p+1}$, we have a correspondence 
\begin{center}
\begin{tikzcd} 
& C_{[0,a]}^{\can} \ar[ld,"p^{\can}_2"'] \ar[rd,"p^{\can}_1"] & \\
X_0(p)_{[0,pa]} & & X_0(p)_{[0,a]}
\end{tikzcd}
\end{center}
where $p_1^{\can}$ is actually an isomorphism. We can think of  this correspondence as the graph of the Frobenius map  $ X_0(p)_{[0,pa]} \rightarrow X_0(p)_{[0,a]}$, sending $(E,H_1)$ to $(E/H_1^{\can}, E[p]/H_1^{\can})$.  We claim that there is  an associated operator:
$$U_p^{\naive, \,\can} : \mathrm{R} \Gamma_{X_0(p)_{[0,a[}} (X_0(p), \omega^k) \longrightarrow \mathrm{R} \Gamma_{X_0(p)_{[0,a[}} (X_0(p), \omega^k).$$

We first observe that $\mathrm{R} \Gamma_{X_0(p)_{[0,a[}} (X_0(p), \omega^k) = \mathrm{R} \Gamma_{X_0(p)_{[0,a[}} (X_0(p)_{[0,a]}, \omega^k)$. We now construct $U_p^{\naive,\,can}$ as the composite: 
\begin{center}
\begin{tikzcd}[row sep=normal, column sep=normal]
\mathrm{R} \Gamma_{X_0(p)_{[0,a[}} \left(X_0(p), \omega^k\right) \arrow[r]
& \mathrm{R} \Gamma_{X_0(p)_{[0,pa[}} \left(X_0(p), \omega^k\right) \ar[r,"(p_2^{\can})^\star"] \arrow[d, phantom, ""{coordinate, name=Z}] &  \mathrm{R} \Gamma_{ C^{\can}_{[0,{a}[} }\left( C^{\can}_{[0, {a}]},  (p_2^{\can})^\star\omega^k\right) \arrow[dl,
rounded corners,
to path={ -- ([xshift=2ex]\tikztostart.east)
|- (Z) [near start]\tikztonodes
-| ([xshift=-2ex]\tikztotarget.west)
-- (\tikztotarget)}] \\
& \mathrm{R} \Gamma_{ C^{\can}_{[0,{a}[} }\left( C^{\can}_{[0, {a}]},  (p_1^{\can})^\star\omega^k\right) \ar[r,"\sim"]  &\mathrm{R} \Gamma_{X_0(p)_{[0,{a}[}} \left(X_0(p), \omega^k\right).
\end{tikzcd}
\end{center}

\begin{prop}\label{prop-U_p=Frob} For $a < \frac{1}{p+1}$, we have $U_p^{\naive,\,\can}  = U_p^{\naive}$ on  $\mathrm{R} \Gamma_{X_0(p)_{[0,a[}} (X_0(p), \omega^k)$.
\end{prop}  

\begin{proof} This amounts to comparing the construction of $U_p^{\naive}$ (given in the proof of Proposition~\ref{lem-construction-U_p^{naive}}) and of $U_p^{\naive,\,\can}$. We see that $p_1(C^{[0,a[}) \subset X_0(p)_{[0,a]}$ and therefore, $$p_1^{-1} p_1(C^{[0,a[}) = (p_1^{-1} p_1(C^{[0,a[}) \cap C_{[0,a]}^{\can}) \coprod (p_1^{-1} p_1(C^{[0,a[}) \cap C_{[0,a]}^{\et}).$$
Moreover, $p_2((p_1^{-1} p_1(C^{[0,a[}) \cap C_{[0,a]}^{\et})) \subset X_0(p)_{[1-a,a]}$ and therefore, $C^{[0,a[} \hookrightarrow (p_1^{-1} p_1(C^{[0,a[}) \cap C_{[0,a]}^{\can})$. 
We have
\[\mathrm{R}\Gamma_{p_1^{-1} p_1(C^{[0,a[})} ( C, p_1^\star \omega^k) = \mathrm{R}\Gamma_{(p_1^{-1} p_1(C^{[0,a[}) \cap C_{[0,a]}^{\can}) } ( C,p_1^\star \omega^k) \oplus  \mathrm{R}\Gamma_{(p_1^{-1} p_1(C^{[0,a[}) \cap C_{[0,a]}^{\et})}( C , p_1^\star\omega^k).\]

Therefore the map $\mathrm{R}\Gamma_{C^{[0,a[}} (C,p_1^\star \omega^k) \rightarrow \mathrm{R}\Gamma_{p_1^{-1} p_1(C^{[0,a[})} ( C, p_1^\star \omega^k)$ factors through the direct factor 
 $\mathrm{R}\Gamma_{(p_1^{-1} p_1(C^{[0,a[}) \cap C_{[0,a]}^{\can}) } ( C, p_1^\star\omega^k)$. 
 \end{proof} 

 \subsubsection{Slopes estimates and the control theorem}  The following lemma is the key technical input to proving Coleman's classicality theorem. 
 
 \begin{lem}\label{lem-slope} For any $a \in ]0,1[ \cap \qq$, \begin{enumerate}
 \item the slopes of $U_p^{\naive}$ on  $\mathrm{R} \Gamma( X_0(p)_{[a,1]}, \omega^k)$ are $\geq 1$. 
\item the slopes of $U_p^{\naive}$ on  $\mathrm{R}\Gamma_{X_0(p)_{[0,a[}} ( X_0(p), \omega^k)$ are $\geq k$. 
\end{enumerate}
\end{lem}

\begin{proof}   We first observe that the finite slope part of the cohomology $\mathrm{R} \Gamma( X_0(p)_{[a,1]}, \omega^k)^{fs}$  is independent of $a \in ]0,1[$ and is therefore supported in degree $0$ by affineness for $a$ close to 1. It follows that we are left to prove that $U_p^{\naive}$ has slopes at least $1$ and this is a $q$-expansion computation. Namely, $U_p^{\naive} (\sum a_n q^n) = p \sum a_{np} q^n$.
For the second cohomology, we again observe that the cohomology is independent of $a \in ]0,1[$. We may therefore suppose that $a$ is small enough and use that $U^{\naive}_p = U_p^{\naive,\,\can}$ (Proposition~\ref{prop-U_p=Frob}).  We define an invertible  sheaf  of $\oscr_{X_0(p)}^+$-modules, $\omega_E^+ \subseteq \omega_E$.   A section $f \in \omega_E(U)$ belongs to $\omega_E^+(U)$ if for any $x \in U$, $f_x$ defines a section of the conormal sheaf on the extension $\tilde{E} \rightarrow \Spec~k(x)^+$ of the semi-abelian scheme $E$ at $x$. Alternatively, we have a specialization morphism $\mathrm{sp} : X_0(p) \rightarrow \mathfrak{X}$ where  $\mathfrak{X}$ is the formal scheme equal to the completion of  the prime-to-$p$ level modular curve $X$ viewed as a scheme over $\Spec~\ZZ_p$. Then $$\omega_E^+ = \mathrm{sp}^{-1} \omega_E \otimes_{\mathrm{sp}^{-1}  \oscr_{\mathfrak{X}}} \oscr_{X_0(p)}^+$$ where $\omega_E$ is the modular sheaf over $\mathfrak{X}$.  We let $\omega^{k,+} = (\omega_E^+)^{\otimes k}$. 

We first claim that for any $s \in \qq$, 
\[\im\left( \HH^i_{X_0(p)_{[0,a[}} \left( X_0(p), \omega^{k,+}\right) \longrightarrow  \HH^i_{X_0(p)_{[0,a[}} \left( X_0(p), \omega^{k}\right)^{ = s}\right)\]  (where the supscript $=s$ means the slope $s$ part for the action of $U_p^{\naive}$)
defines a lattice in  $\HH^i_{X_0(p)_{[0,a[}} ( X_0(p), \omega^{k})^{ = s}$ (an open and bounded submodule).   We can indeed represent the cohomology $\mathrm{R}\Gamma_{X_0(p)_{[0,a[}} ( X_0(p), \omega^{k})$ by  the \v{C}ech complex $C^\bullet$ relative to some open covering $\mathcal{U}$, and we can lift the $U_p$-operator to a compact operator $\tilde{U}_p$ on the complex. Note that $C^\bullet$ is a complex of Banach modules. 

 The map from \v{C}ech cohomology with respect to $\mathcal{U}$ to cohomology
 \[\check{\HH}^i_{ \mathcal{U}, X_0(p)_{[0,a[}} \left( X_0(p), \omega^{k,+}\right) \longrightarrow \HH^i_{X_0(p)_{[0,a[}} \left( X_0(p), \omega^{k,+}\right)\] has kernel and cokernel of bounded torsion by \cite[Lemma~3.2.2]{pilloniHidacomplexes}.  It suffices to prove that 
 \[\im \left( \check{\HH}^i_{\mathcal{U}, X_0(p)_{[0,a[}} \left( X_0(p), \omega^{k,+}\right) \longrightarrow  \HH^i_{X_0(p)_{[0,a[}} \left( X_0(p), \omega^{k}\right)^{ = s}\right)\]
defines an open and bounded submodule in $\HH^i_{X_0(p)_{[0,a[}} ( X_0(p), \omega^{k})^{ = s}$. The \v{C}ech cohomology
\[\check{\HH}^i_{ \mathcal{U}, X_0(p)_{[0,a[}} \left( X_0(p), \omega^{k,+}\right)\] is obtained by taking the cohomology of an open and bounded sub-complex $C^{+, \bullet} \subset C^{\bullet}$. The image of $C^{+, \bullet}$ in $C^{\bullet, = s}$ under the continuous projection $C^{\bullet} \rightarrow C^{\bullet, =s}$ (the target is now a complex of finite dimensional vector spaces) is again open and bounded and the claim follows.  

Moreover,   over $C_{[0,a[}^{\can}$, we have a universal isogeny  which gives an isomorphism, $p_2^\star \omega \rightarrow p_1^\star \omega$, for which $p p_1^\star \omega^+ \subset p_2^\star \omega^+ \subset p^{1- \frac{a}{p}} p_1^\star \omega^+$. We deduce that 
$U_p^{\naive}$  induces a  map  $$\mathrm{R}\Gamma_{X_0(p)_{[0,a[}} ( X_0(p), \omega^{k,+}) \longrightarrow \mathrm{R}\Gamma_{X_0(p)_{[0,a[}} ( X_0(p), p^{k(1-\frac{a}{p})}\omega^{k,+})$$ if $k \geq 0$, and $$\mathrm{R}\Gamma_{X_0(p)_{[0,a[}} ( X_0(p), \omega^{k,+}) \longrightarrow \mathrm{R}\Gamma_{X_0(p)_{[0,a[}} ( X_0(p), p^{k}\omega^{k,+})$$ if $k \leq 0$. We deduce that $p^{-k(1- \frac{a}{p})} U_p^{\naive}$ if $k \geq 0$ and $p^{-k} U_p^{\naive}$  if $k \leq 0$ stabilize a lattice in the cohomology, and therefore have only non-negative slope. The lemma follows. 
\end{proof}

\begin{rem} Lemma~3.2.2 in \cite{pilloniHidacomplexes} depends on the main result of \cite{MR488562}, which in turn is a key technical ingredient in \cite{MR2219265}. 
\end{rem}

We now define $U_p = p^{-\inf\{1,k\}} U_p^{\naive}$ and we get:

\begin{thm}
\leavevmode
\begin{enumerate}
\item $U_p$ has slopes $\geq0$ on $\mathrm{R}\Gamma(X_0(p), \omega^k)$,
\item  For any $a \in ]0,1[ \cap \qq$, the map $\mathrm{R\Gamma}(X_0(p), \omega^k)^{< k-1} \rightarrow \mathrm{R\Gamma}(X_0(p)_{[a,1]}, \omega^k)^{< k-1}$ is a quasi-isomorphism,
\item  For any $a \in ]0,1[ \cap \qq$, the map $\mathrm{R\Gamma}_{X_{[0,a[}}(X_0(p), \omega^k)^{< 1-k} \rightarrow \mathrm{R\Gamma}(X_0(p), \omega^k)^{< 1-k}$ is a quasi-isomorphism. 
\end{enumerate}
\end{thm}
\begin{proof} We consider the triangle 
\[\begin{tikzcd}
\mathrm{R}\Gamma_{X_0(p)_{[0,a[}} \left(X_0(p), \omega^k\right) \ar[r] &  \mathrm{R}\Gamma(X_0(p), \omega^k) \ar[r] & \mathrm{R}\Gamma(X_0(p)_{[a,1]}, \omega^k)  \ar[r,"+1"] & {}
\end{tikzcd}\]
on which $U_p$ acts  equivariantly and apply the slope estimates of Lemma~\ref{lem-slope}. 
\end{proof}

\subsubsection{Comparison with spherical level}  For $a > \frac{1}{p+1}$, the map $p_1 : X_0(p)_{[a,1]} \rightarrow X_{1-a}$ is an isomorphism by Theorem~\ref{thm-can} and therefore the pull back map $\mathrm{R}\Gamma (X_{1-a}, \omega^k) \rightarrow \mathrm{R}\Gamma (X_0(p)_{[a,1]}, \omega^k)$ is a quasi-isomorphism.

There is an analogous statement for the cohomology with support that we now explain.  We first introduce a cohomology with support at spherical level.  For $a<1$ we define $X_{<a}\subseteq X$ as the complement in $X$ of the quasi-compact open whose rank one points are $\{x\in X^{\mathrm{rk 1}}\mid\mathrm{Hdg}(x)\geq a\}$ (see Section~\ref{sec-can-sub}).  We may then form the cohomology with support $\mathrm{R}\Gamma_{X_{<a}}(X,\omega^k)$.

Let $a< \frac{1}{p+1}$.  We have a Frobenius map $F : X_0(p)_{[0,a]} \rightarrow X_0(p)_{[0,pa]}$ given by
\[(E,H_1) \longmapsto \left(E/H_1^{\can}, E[p]/H_1^{\can}\right).\]
There is similarly a Frobenius map $F :  {X}_{\frac{a}{p}}  \rightarrow X_a$, given by $E \mapsto E/H_1^{\can}$. 

The Frobenius map fits into the following diagram:
\[
\begin{tikzcd}[column sep=large]
X_0(p)_{[0,a]} \ar[r,"F"]  \ar[d,"p_1"'] & X_0(p)_{[0,pa]} \ar[d,"p_1"] \\
{X}_{\frac{a}{p}} \ar[ru] \ar[r,"F"'] & {X}_{a}
\end{tikzcd}
\]
where the diagonal map is given by $E \mapsto (E/H_1^{\can}, E[p]/H_1^{\can})$.

We can define an endomorphism $F^\star$ of $\mathrm{R}\Gamma_{X_{<\frac{a}{p}}}(X,\omega^k)$ as the composite
\begin{center}
\begin{tikzcd}[row sep=normal, column sep=normal]
\mathrm{R}\Gamma_{X_{<\frac{a}{p}}}\left(X,\omega^k\right) \ar[r] & \mathrm{R}\Gamma_{X_{<a}}\left(X,\omega^k\right)\ar[r,"\sim"]\arrow[d, phantom, ""{coordinate, name=Z}]
& \mathrm{R}\Gamma_{X_{<a}}\left(X_a,\omega^k\right)  \arrow[dl,
rounded corners,
to path={ -- ([xshift=2ex]\tikztostart.east)
|- (Z) [near start]\tikztonodes
-| ([xshift=-2ex]\tikztotarget.west)
-- (\tikztotarget)}] &\\
& \mathrm{R}\Gamma_{X_{<\frac{a}{p}}}\left(X_{\frac{a}{p}},F^*\omega^k\right)\ar[r]& \mathrm{R}\Gamma_{X_{<\frac{a}{p}}}\left(X_{\frac{a}{p}},\omega^k\right) \ar[r,"\sim"]  &\mathrm{R}\Gamma_{X_{<\frac{a}{p}}}\left(X,\omega^\kappa\right)
\end{tikzcd}
\end{center}
where the first map is the change of support for the inclusion $X_{\frac{a}{p}}\subseteq X_a$, which is compact as in the proof of Proposition~\ref{lem-construction-U_p^{naive}}, the third map is pullback by $F$, and the fourth map is induced by the map $F^\star\omega^k\rightarrow\omega^k$ which comes from the isogeny $E\rightarrow E/H_1^{\can}$.  The same construction applied to $X_0(p)_{[0,a]}$ yields the operator $U_p^{\naive,\,\can}=U_p^{\naive}$ of Section~\ref{sec-up-frob} as the correspondence $C_{[0,a]}^{\can}$ is the graph of $F$.

\begin{prop} The pull back map  $p_1^\star:\mathrm{R}\Gamma_{X_{<\frac{a}{p}}}(X, \omega^k) \rightarrow  \mathrm{R}\Gamma_{X_0(p)_{[0,a[}}(X_0(p), \omega^k)$ induces a quasi-isomorphism on the finite slope parts for $F^\star$ and $U_p^{\naive}$
\end{prop}

\begin{proof}
This follows from the existence of a commutative diagram
\[
\begin{tikzcd}[column sep=large]
\mathrm{R}\Gamma_{X_{<\frac{a}{p}}}(X, \omega^k)\ar[r,"p_1^\star"]  \ar[d,"F^\star"'] & \mathrm{R}\Gamma_{X_0(p)_{[0,a[}}(X_0(p), \omega^k)\ar[dl] \ar[d,"U_p^{\naive}"] \\
\mathrm{R}\Gamma_{X_{<\frac{a}{p}}}(X, \omega^k) \ar[r,"p_1^\star"'] & \mathrm{R}\Gamma_{X_0(p)_{[0,a[}}(X_0(p), \omega^k)
\end{tikzcd}
\]
which is deduced from the definitions and the commutative diagram above.
\end{proof}

\subsection{$p$-adic variation} We now consider the problem of  interpolation of these cohomologies. 
\subsubsection{Reduction of the torsor $\omega_E$}   
We recall here a construction from \cite{MR3097946} and \cite{MR3265287}. 
We let $\mathcal{T} = \{ w \in \omega_E,~\omega \neq 0\}$ be the $\mathbb{G}_m$-torsor associated to $\omega_E$. We let $\mathbb{G}_a^+ = \Spa (\qq_p \langle T \rangle, \ZZ_p \langle T \rangle )$ be the unit ball, with its additive analytic group  structure. We have subgroups $\ZZ_p^\times(1 + p^{r} \mathbb{G}_a^+) \hookrightarrow \mathbb{G}_m $ for any positive rational number $r$.    

\begin{prop} Let $n \in \ZZ_{\geq 1}$. Let $v < \frac{1}{p^{n-1}(p-1)}$. The $\mathbb{G}_m$-torsor $\mathcal{T}\times_X X_v \rightarrow X_v$ has a natural 
reduction to a  $ \ZZ_p^\times(1 + p^{n-v\frac{p^n}{p-1}} \mathbb{G}_a^+)$-torsor denoted $\mathcal{T}_v$.
\end{prop} 

\begin{proof} We denote by $\omega^+_E \subset \omega_E$ the locally free sheaf of integral relative differential forms. 
For $v < \frac{1}{p^{n-1}(p-1)}$, there is a canonical subgroup of level $n$,  $H^{\can}_n$ over $X_v$.  The isogeny $E \rightarrow E/H^{\can}_n$   yields a map $\omega_{E/H^{\can}_n}^+ \rightarrow \omega_{E}^+$, with cokernel $\omega_{H_n}^+$. 
The surjective map $r : \omega_{E}^+ \rightarrow \omega_{H^{\can}_n}^+$ induces an isomorphism: 
$$ \omega_{E}^+/p^{n-v\frac{p^n-1}{p-1}} \xrightarrow{\sim} \omega_{H^{\can}_n}^+/p^{n-v\frac{p^n-1}{p-1}}.$$

There is a Hodge-Tate map $\mathrm{HT}:(H^{\can}_n)^D \rightarrow \omega_{H^{\can}_n}^+$ (of sheaves on the \'etale site, see \cite[p.~117]{MR0347836}), and its linearization $\mathrm{HT} \otimes 1  : (H^{\can}_n)^D \otimes \oscr_{X_v}^+ \rightarrow \omega_{H^{\can}_n}^+$  has cokernel killed by $p^{\frac{v}{p-1}}$.  We have a diagram:
\begin{eqnarray*}
\xymatrix{ & \omega_{E}^+ \ar[d]^{r} \\
(H^{\can}_n)^D \ar[r]^{\mathrm{HT}} & \omega_{H_n}^+}
\end{eqnarray*}

We now introduce a  modification of $\omega_{E}^+$:  
let $\omega_{E}^\sharp  = \{ w \in \omega_{E}^+,~r(w) \in \im(\mathrm{HT} \otimes 1) \} \subset \omega_{E}^+$. This is a locally free sheaf of $\oscr_{X_v}^+$-modules on the \'etale site. 
The Hodge-Tate map induces an isomorphism:
$$ \mathrm{HT}_v : (H^{\can}_n)^D \otimes \oscr_{X_v}^+/p^{n-v\frac{p^n}{p-1}} \longrightarrow \omega_{E}^\sharp/p^{n-v\frac{p^n}{p-1}}.$$
We let $\mathcal{T}_v$ be the torsor under the group $\ZZ_p^\times(1 + p^{n-v\frac{p^n}{p-1}} \mathbb{G}_a^+)$ defined by 
$$\mathcal{T}_v = \left\{  \omega \in   \omega_E^\sharp,~\exists P \in (H^{\can}_n)^D, ~p^{n-1}P \neq 0,  \mathrm{HT}_v(P) =  \omega \mod  p^{n-v\frac{p^n}{p-1}} \right\}.$$
We have a natural  map $\mathcal{T}_v \hookrightarrow \mathcal{T}$, equivariant for the  analytic group  map: $ \ZZ_p^\times(1 + p^{n-v\frac{p^n}{p-1}} \mathbb{G}_a^+) \rightarrow \mathbb{G}_m$. 
\end{proof}

\subsubsection{Interpolation of the sheaf}

We let $\mathcal{W} = \Spa (\Lambda, \Lambda) \times \Spa (\qq_p, \ZZ_p)$ be the weight space. We let $\kappa^{\un} : \ZZ_p^\times \rightarrow \oscr_{\mathcal{W}}^\times$ be the universal character. 

 We can write $\mathcal{W}$ as an increasing union of affinoids $\mathcal{W} = \cup_{0< r <1} \mathcal{W}_r$ where $r \in \qq \cap ]0,1[$ and each $\mathcal{W}_r$ is a finite union of balls of radius $r$. Over each $\mathcal{W}_r$, there is $t(r) \in \qq_{>0}$ such that the universal character extends to a character $\kappa^{\un} :  \ZZ_p^\times(1+p^{t(r)} \oscr_{\mathcal{W}}^+) \rightarrow \oscr_{\mathcal{W}}^\times$. 
 
We now fix $r$ and  we choose $v$ small enough and $n$ large enough such that $t(r) \leq n- v \frac{p^n}{p-1}$ and we define a locally free sheaf $\omega^{\kappa^{\un}}$ over $X_v \times \mathcal{W}_r$: 
\[\omega^{\kappa^{\un}} = (\oscr_{\mathcal{T}_v} \otimes \oscr_{\mathcal{W}_r})^{\ZZ_p^\times(1 + p^{n- v \frac{p^n}{p-1}} \oscr_{X_v \times \mathcal{W}_r}^+)}.\]
 
Since $\mathcal{T}_v \rightarrow X_v$ is an \'etale torsor, the sheaf $\omega^{\kappa^{\un}}$ is a locally free sheaf of $\oscr_{X_v \times \mathcal{W}_r}$-modules in the \'etale topology. It is actually a locally free sheaf of $\oscr_{X_v\times \mathcal{W}_r}$-modules in the Zariski topology by the main result of \cite{MR1603849}. 
 
 \subsubsection{Interpolation of the cohomology}
 
 For $a \in ]0, \frac{v}{p}]$, we have a  map $p_1 : X_0(p)_{[0,a]} \rightarrow X_v$ and we can therefore pull back the sheaf $\omega^{\kappa^{\un}}$ to an invertible sheaf over $X_0(p)_{[0,a]} \times \mathcal{W}_r$. 
 
 If $a \in [1-v,1[$,  we have a  map $p_1 : X_0(p)_{[a,1]} \rightarrow X_v$ and we can  pull back the sheaf $\omega^{\kappa^{\un}}$ to an invertible sheaf over $X_0(p)_{[a,1]} \times \mathcal{W}_r$.

 We consider the cohomologies: 
 
 \begin{enumerate}
 \item  $\mathrm{R}\Gamma_{X_{0}(p)_{[0,a[}}(X_0(p)  , \omega^{\kappa^{\un}})$ 
 \item  $\mathrm{R}\Gamma(X_0(p)_{[a,1]}  , \omega^{\kappa^{\un}})$.
 \end{enumerate}
 
  Note that the first cohomology group is well defined because
  \[\mathrm{R}\Gamma_{X_{0}(p)_{[0,a[}}(X_0(p)  , \omega^{\kappa^{\un}}) = \mathrm{R}\Gamma_{X_{0}(p)_{[0,a[}}(X_0(p)_{[0,a]}  , \omega^{\kappa^{\un}}).\]
 
 These cohomologies belong to the category of perfect complexes of Banach spaces over $\oscr_{\mathcal{W}_r}$ which is the homotopy category of the category of bounded complexes of projective Banach modules over $\oscr_{\mathcal{W}_r}$.

 \subsubsection{The $U_p$-operator on $\mathrm{R}\Gamma(X_0(p)_{[a,1]}  , \omega^{\kappa^{\un}})$}\label{sect-defining-U_p}
 We need to consider the $U_p$-correspondence on $X_0(p)_{[a,1]}$  for $a \geq 1-v > \frac{1}{p+1}$, where by Proposition~\ref{prop-dyn-1}, it induces to a correspondence
\[
 \xymatrix{ & C_{[a,1]} \ar[rd]^{p_1} \ar[ld]_{p_2}& \\
 X_0(p)_{[ \frac{p-1}{p}+\frac{a}{p}, 1]} & & X_0(p)_{[{a}, 1]}}
\]
 \begin{lem} There is a natural isomorphism $p_2^\star \omega^{\kappa^{\un}} \rightarrow p_1^\star \omega^{\kappa^{\un}}$, and we can define a cohomological correspondence $U_p : p_{1\star} p_2^\star \omega^{\kappa^{\un}} \rightarrow \omega^{\kappa^{\un}}$ which specializes in weight $k \geq 1$ to $U_p$. 
 \end{lem}
 \begin{proof} Over $C_{[a,1]}$, the universal isogeny $\pi : p_1^\star {E} \rightarrow p_2^\star E$ induces an isomorphism on canonical subgroups $p_1^\star H^{\can}_n \simeq p_2^\star H^{\can}_n$, and therefore there is a canonical isomorphism:
 
\[
\xymatrix{ p_2^\star \omega_{E}^\sharp \ar[r] \ar[d] & p_1^\star \omega_{E}^\sharp  \ar[d] \\
p_2^\star \omega_{E}^\sharp/p^{n-v \frac{p^n}{p-1}} \ar[r]  & p_1^\star \omega_{E}^\sharp /p^{n-v \frac{p^n}{p-1}} \\
p_2^\star (H^{\can}_n)^D \ar[r]_{\pi^D} \ar[u]^{\mathrm{HT}} & p_1^\star (H^{\can}_n)^D \ar[u]_{\mathrm{HT}}}
\]

This clearly induces an isomorphism $p_1^\star \mathcal{T}_v \rightarrow p_2^\star \mathcal{T}_v$ and from this we get an isomorphism:
\[p_2^\star \omega^{\kappa^{\un}} \longrightarrow p_1^\star \omega^{\kappa^{\un}}\]
which specializes to the natural isomorphism $p_2^\star \omega^k \rightarrow p_1^\star \omega^k$ at weight $k$.

We now define
\[U_p : p_{1\star} p_2^\star \omega^{\kappa^{\un}} \longrightarrow  p_{1\star} p_1^\star \omega^{\kappa^{\un}} \xrightarrow{\frac{1}{p} \mathrm{\Tr}_{p_1}} \omega^{\kappa^{\un}}.\]

\end{proof}

\begin{coro} The operator $U_p$ is compact on $\mathrm{R}\Gamma(X_0(p)_{[a,1]}, \omega^{\kappa^{\un}})$.
\end{coro}
\begin{proof} The operator $U_p$ factors as:
$$\mathrm{R}\Gamma\left(X_0(p)_{[a,1]}, \omega^{\kappa^{\un}}\right) \longrightarrow \mathrm{R}\Gamma\left(X_0(p)_{[\frac{a}{p},1]}, \omega^{\kappa^{\un}}\right)  \longrightarrow \mathrm{R}\Gamma\left(X_0(p)_{[a,1]}, \omega^{\kappa^{\un}}\right).$$
\end{proof}

\subsubsection{The $U_p$-operator on $\mathrm{R}\Gamma_{X_{0}(p)_{[0,a[}}(X_0(p)  , \omega^{\kappa^{\un}})$}\label{sect-defining-U_p2}
 
We need to consider the $U_p$-correspondence on $X_0(p)_{[0,a]}$ where actually only the canonical part of the correspondance is relevant by Proposition~\ref{prop-U_p=Frob}, and therefore it reduces to a correspondence: 
\[
 \xymatrix{ & C_{[0,a]}^{\can} \ar[rd]^{p^{\can}_1} \ar[ld]_{p^{\can}_2}& \\
 X_0(p)_{[0, pa]} & & X_0(p)_{[0, a]}}
\]

\begin{lem} There is a natural isomorphism $(p^{\can}_2)^\star \omega^{\kappa^{\un}} \rightarrow (p^{\can}_1)^\star \omega^{\kappa^{\un}}$, and a cohomological correspondence $U_p : (p^{\can}_1)_\star (p^{\can}_2)^\star \omega^{\kappa^{\un}} \rightarrow \omega^{\kappa^{\un}}$ which specializes in weight $k \leq 1$ to $U^{\can}_p$. 
\end{lem}

\begin{proof}  Over $C^{\can}_{[0,a]}$, the dual universal isogeny $\pi^D: (p^{\can}_2)^\star {E} \rightarrow (p^{\can}_1)^\star E$ induces an isomorphism on canonical subgroups $(p^{\can}_2)^\star H^{\can}_n \simeq (p^{\can}_1)^\star H^{\can}_n$, and therefore there is a canonical isomorphism:
\[
\xymatrix{ (p^{\can}_1)^\star \omega_{E}^\sharp \ar[r]^{(\pi^D)^\star} \ar[d] & (p^{\can}_2)^\star \omega_{E}^\sharp  \ar[d] \\
(p^{\can}_1)^\star \omega_{E}^\sharp/p^{n-v \frac{p^n}{p-1}} \ar[r]  & (p^{\can}_2)^\star \omega_{E}^\sharp /p^{n-v \frac{p^n}{p-1}} \\
(p^{\can}_1)^\star (H^{\can}_n)^D \ar[r]_{\pi} \ar[u]^{\mathrm{HT}} & (p^{\can}_2)^\star (H^{\can}_n)^D \ar[u]_{\mathrm{HT}}}
\]

This clearly induces an isomorphism $(p^{\can}_2)^\star \mathcal{T}_v \rightarrow (p^{\can}_1)^\star \mathcal{T}_v$ and from this we get an isomorphism:
$$ (p^{\can}_2)^\star \omega^{\kappa^{\un}} \longrightarrow (p^{\can}_1)^\star \omega^{\kappa^{\un}}$$ or rather more naturally its inverse: 
$$ (p^{\can}_1)^\star \omega^{\kappa^{\un}} \longrightarrow (p^{\can}_2)^\star \omega^{\kappa^{\un}}$$
 which specializes to the natural isomorphism $(p^{\can}_1)^\star \omega^k \rightarrow (p^{\can}_2)^\star \omega^k$ given by $(\pi^D)^\star$, and its inverse $(p^{\can}_2)^\star \omega^k \rightarrow (p^{\can}_1)^\star \omega^k$ is $p^{-k} (\pi^\star)$.

Recall that $p_1$ is an isomorphism, and we therefore get  $U_p : (p^{\can}_1)_\star (p^{\can}_2)^\star \omega^{\kappa^{\un}} \rightarrow   \omega^{\kappa^{\un}}$ which specializes in weight $k \leq 1$ to $p^{-k}U_p^{\naive,\,\can}= U^{\can}_p$. 
\end{proof}

\begin{rem} We therefore find that this is really $U_p$ and not $U_p^{\naive}$ that can be interpolated over the weight space. 
\end{rem} 
\begin{coro} The $U_p$-operator is compact on $\mathrm{R}\Gamma_{X_0(p)_{[0,a[}}(X_0(p), \omega^{\kappa^{\un}})$. 
\end{coro}

\begin{proof} The operator $U_p$ factors as:
$$\mathrm{R}\Gamma_{X_0(p)_{[0,a[}}\left(X_0(p), \omega^{\kappa^{\un}}\right) \longrightarrow \mathrm{R}\Gamma_{X_0(p)_{[0,pa[}}\left(X_0(p), \omega^{\kappa^{\un}}\right)  \longrightarrow \mathrm{R}\Gamma_{X_0(p)_{[0,a[}}\left(X_0(p), \omega^{\kappa^{\un}}\right).$$
\end{proof}

\subsection{Construction of eigencurves} We use these cohomologies to construct the eigencurve, following the method of \cite{1997InMat.127..417C}.

\subsubsection{ First construction}\label{first-construction}
The cohomology $\mathrm{R}\Gamma(X_0(p)_{[a,1]}, \omega^{\kappa^{\un}})$ is concentrated in degree $0$, and is represented by $\HH^0(X_0(p)_{[a,1]}, \omega^{\kappa^{\un}})$ which is a projective Banach module over $\oscr_{\mathcal{W}_r}$ (see \cite[Corollary~5.3]{MR3097946}). 

The $U_p$-operator acts compactly on this space. We let $\mathcal{P} \in \oscr_{\mathcal{W}_r}[[T]]$ be the characteristic series of $U_p$. This is an entire series. We let $\mathcal{Z} = V(\mathcal{P}) \subset \mathbb{G}_m^{\mathrm{an}} \times \mathcal{W}_r$. We have a weight map $\pi : \mathcal{Z} \rightarrow \mathcal{W}_r$ which is quasi-finite, partially proper, and locally on the source and the target a finite flat map.

 Over $\mathcal{Z}$ we have a coherent sheaf $\mathcal{M}$ which is the universal generalized eigenspace. This is a locally free as a $\pi^{-1} \oscr_{\mathcal{W}_r}$-module and for any $x = (\kappa, \alpha) \in \mathcal{Z}$,  $x^\star \mathcal{M} = \HH^0(X_0(p)_{[a,1]}, \omega^{\kappa})^{U_p= \alpha^{-1}}$. 

We let $\oscr_{\mathcal{C}} \subseteq \mathrm{End}_{\mathcal{Z}}(\mathcal{M})$ be the subsheaf of $\oscr_{\mathcal{Z}}$-modules generated by the Hecke operators of level prime to $Np$, and we let $\mathcal{C} \rightarrow \mathcal{Z}$ be the associated analytic space.

The construction of $(\mathcal{M}, \mathcal{C}, \mathcal{Z})$ is compatible when $r$ changes (and does not depend on auxiliary choices like $a$, $v$, $n\ldots$ ). We  now let $r$ tend to  $1$, glue everything,  and with a slight abuse of notation we have $\mathcal{C} \rightarrow \mathcal{Z} \rightarrow \mathcal{W}$ and a coherent sheaf $\mathcal{M}$ over $\mathcal{C}$. This is the eigencurve of \cite{MR1696469}. 

\subsubsection{Second construction} We can perform a similar construction, using instead the cohomology $\mathrm{R}\Gamma_{X_0(p)_{[0,a[}}(X_0(p), \omega^{2-\kappa^{\un}}(-D))$ where $\omega^{2-\kappa^{\un}}(-D) = (\omega^{\kappa^{\un}})^\vee \otimes \omega^2(-D)$, as well as the operator $\langle p \rangle^{-1} U_p$. The introduction of this twist is motivated by Serre duality (see Section \ref{section-classical} below). 
Recall that the character $\ZZ_p^\times \rightarrow \Lambda^\times$, $t \mapsto t^2 \kappa^{\un}(t)^{-1}$ induces an automorphism $d : \Lambda \rightarrow \Lambda$, and therefore: $$\mathrm{R}\Gamma_{X_0(p)_{[0,a[}}(X_0(p), \omega^{2-\kappa^{\un}}(-D)) = \mathrm{R}\Gamma_{X_0(p)_{[0,a[}}(X_0(p), \omega^{\kappa^{\un}}(-D)) \otimes^L_{\oscr_{\mathcal{W}_r},d} \oscr_{\mathcal{W}_r}.$$

 This cohomology  is   a perfect complex of projective Banach modules over $\oscr_{\mathcal{W}_r}$ and for any morphism $\Spa(A, A^+) \rightarrow \mathcal{W}_r$, the cohomology $\mathrm{R}\Gamma_{[0,a[}(X_0(p), \omega^{2-\kappa^{\un}}(-D))\hat{\otimes} A)$ is supported in degree $1$. 

We  choose a representative $N^\bullet$ for this cohomology, as well as a compact representative  $\widetilde{\langle p \rangle^{-1} U_p}$ representing the action of $\langle p \rangle^{-1} U_p$. We let ${\mathcal{Q}}_i$ be the characteristic series of $\widetilde{\langle p \rangle^{-1} U_p}$ acting on $N^i$.
We let $\tilde{\mathcal{Q}} = \prod \mathcal{Q}_i$  and we let $\tilde{\mathcal{X}} = V(\tilde{\mathcal{Q}}) \subseteq \mathbb{G}_m^{\mathrm{an}} \times \mathcal{W}_r$. We have a weight map $\pi : \tilde{\mathcal{X}} \rightarrow \mathcal{W}_r$ which is quasi-finite  locally on the source and the target  and a bounded complex of coherent sheaves ``generalized eigenspaces'' $\tilde{\mathcal{N}}^\bullet$ over $\tilde{\mathcal{X}}$. This  is a perfect complex of finite projective $\pi^{-1} \oscr_{\mathcal{W}_r}$-modules. Moreover, $\tilde{\mathcal{N}}^\bullet$ has cohomology  only in degree $1$, and we deduce that $\HH^1(\mathcal{N}^\bullet) = \mathcal{N}$ is a locally free $\pi^{-1} \oscr_{\mathcal{W}_r}$-module. We let $\mathcal{Q} = V(\prod_i \mathcal{Q}_i^{(-1)^i} )$ and we set $\mathcal{X} = V( \mathcal{Q}) \subseteq \tilde{\mathcal{X}}$. The module $\mathcal{N}$ is supported on $\mathcal{X}$. We let $\oscr_{\mathcal{D}} \subseteq \mathrm{End}_{\mathcal{X}}(\mathcal{N})$ be the subsheaf generated by the Hecke algebra of prime to $Np$ level, and we let $\mathcal{D} \rightarrow \mathcal{X}$ be the associated analytic space. We can now let $r$ tend to  $1$, and we have $\mathcal{D} \rightarrow \mathcal{X} \rightarrow \mathcal{W}$ and the sheaf $\mathcal{N}$ over $\mathcal{D}$. This is a second eigencurve.

\subsection{The duality pairing}  In this last section, we prove that $\mathcal{Z}$ and $\mathcal{X}$ are canonically isomorphic, that $\mathcal{M}$ and $\mathcal{N}$ are canonically dual to each other and that $\mathcal{C}$ and $\mathcal{D}$ are canonically identified under the pairing between $\mathcal{M}$ and $\mathcal{N}$.   \subsubsection{Preliminaries}\label{sect-prelimi} We are going to use the theory of dagger spaces \cite{MR1739729}. Let $X^\dagger$ be a dagger space over $\Spa(\qq_p, \ZZ_p)$, smooth  of pure relative dimension $d$.   Let $\mathscr{F}$ be a coherent sheaf on $X^\dagger$. Then one can define the cohomology groups $\HH^i(X^\dagger, \mathscr{F})$ and $\HH^i_c(X^\dagger, \mathscr{F})$. Moreover, these cohomology groups carry canonical topologies (\cite[1.6]{MR1168350}).

By \cite[Theorem~4.4]{MR1739729}, there is a residue map: $$\mathrm{res}_X : \HH^d_c( X^\dagger, \Omega^d_{X/\qq_p}) \longrightarrow \qq_p.$$

This residue map has the following two important properties. Let $f : Y^\dagger \rightarrow X^\dagger$ be an open immersion. Then the diagram: 
\[
\xymatrix{ \HH^d_c( Y^\dagger, \Omega^d_{Y/\qq_p}) \ar[r]^{f_\star} \ar[rd]_{\mathrm{res}_Y } &  \HH^d_c( X^\dagger, \Omega^d_{X/\qq_p}) \ar[d]^{\mathrm{res}_X} \\ &  \qq_p}
\]
  is commutative. See \cite[Corollary~4.2.12]{MR1464367} (although the author is working here in the ``dual'' setting of wide open spaces)). 

Let $ f : Y^\dagger \rightarrow X^\dagger$ be a finite flat map. Then the diagram: 
\begin{equation}\label{diagram-trace}
\begin{gathered}
\xymatrix{ \HH^d_c( Y^\dagger, \Omega^d_{Y/\qq_p}) \ar[r]^{\mathrm{tr}_f} \ar[rd]_{\mathrm{res}_Y } &  \HH^d_c( X^\dagger, \Omega^d_{X/\qq_p}) \ar[d]^{\mathrm{res}_X} \\ &  \qq_p}
\end{gathered}
\end{equation}
is commutative. See \cite[Corollary~4.2.11]{MR1464367} (although the author is working here again in the ``dual'' setting of wide open spaces).

Let $\mathscr{F}$ be a  locally free sheaf of finite rank over $X^\dagger$ which is assumed to be affinoid, smooth of pure dimension $d$.  We let $\mathbf{D}(\mathscr{F}) = \mathscr{F}^\vee \otimes \Omega^d_{X/\qq_p}$. 

Then the residue map induces a perfect pairing (\cite[Theorem~4.4]{MR1739729}):
$$ \HH^0(X^\dagger, \mathscr{F}) \times  \HH^d_c(X^\dagger, \mathbf{D}(\mathscr{F})) \longrightarrow \qq_p$$
for which both spaces are strong duals of each other. 

\begin{rem}  Since the topological vector space $ \HH^0(X^\dagger, \mathscr{F})$ is a compact inductive limit of Banach spaces, we deduce  from the theorem that   the topological  vector space $ \HH^d_c(X^\dagger, \mathbf{D}(\mathscr{F}))$ is a compact projective limit of Banach spaces. 
\end{rem}

\subsubsection{The classical pairing}\label{section-classical}
We denote by $w$ the Atkin-Lehner involution over $X_0(p)$ and by $\langle p \rangle$ the diamond operator given by multiplication by $p$. We recall that $w \circ w = \langle p \rangle$. We recall that $C$ is the Hecke correspondence underlying $U_p$ (see Section~\ref{sect-corres-Up}). We can think of it as the moduli space of $(E, H, H_1)$ where $H, H_1$ are distinct subgroups of $E[p]$. 
We have the projection $p_1 (E,H, H_1) = (E, H_1)$. We also have a projection $p_2(E,H, H_1) = (E/H, E[p]/H)$. Exchanging the roles of $H$ and $H_1$ yields an automorphism $\iota : C \rightarrow C$ and we let $q_i = p_i \circ \iota$. 
Now one checks easily that $w \circ p_1 = q_2$ and $w \circ p_2 = \langle p \rangle \circ q_1$. 
We have a residue map $\HH^1(X_0(p), \Omega^1_{X_0(p)/\qq_p}) \rightarrow \qq_p$ and there is a perfect pairing:
$$\langle -,- \rangle_0 : \HH^0(X_0(p), \omega^k) \times \HH^1(X_0(p), \omega^{2-k}(-D)) \rightarrow \qq_p$$where we use the Kodaira--Spencer isomorphism $\Omega^1_{X_0(p)/\qq_p} \simeq \omega^2(-D)$. 
We modify this pairing, and set: $$\langle -,- \rangle = \langle -, w^\star (-) \rangle_0.$$

\begin{lem} For any $(f, g) \in \HH^0(X_0(p), \omega^k) \times \HH^1(X_0(p), \omega^{2-k}(-D))$, we have: 
$\langle U_p f , g \rangle=\langle f , \langle p\rangle^{-1}U_p g \rangle$.
\end{lem}

\begin{proof} For any $(f, g) \in \HH^0(X_0(p), \omega^k) \times \HH^1(X_0(p), \omega^{2-k}(-D))$, we have: 
$\langle U^{\naive}_pf ,  g \rangle_0 = \langle  f , (U^{\naive}_p)^t g \rangle_0$ where $(U^{\naive}_p)^t$ is the operator associated to the transpose of $C$, and is obtained as follows:
\[\mathrm{R}\Gamma(X_0(p), \mathbf{D}(\omega^k)) \xrightarrow{p_1^\star} \mathrm{R}\Gamma(C , p_1^\star \mathbf{D}( \omega^k)) \longrightarrow \mathrm{R}\Gamma(C , p_2^\star \mathbf{D}( \omega^k)) \xrightarrow{\mathrm{tr}_{p_2}} \mathrm{R}\Gamma(X_0(p), \mathbf{D}(\omega^k)).\]
We observe that the determination of the adjoint of $U_p^{\naive}$ as the operator associated to the transpose of $C$ uses the  compatibility property of  diagram \ref{diagram-trace}. We have a commutative diagram: 
\[
\xymatrix{ \mathrm{R}\Gamma(X_0(p), \mathbf{D}(\omega^k)) \ar[r]^{p_1^\star} & \mathrm{R}\Gamma(C , p_1^\star \mathbf{D}( \omega^k))  \ar[r] & \mathrm{R}\Gamma(C , p_2^\star \mathbf{D}( \omega^k)) \ar[r]^{\mathrm{tr}_{p_2}} & \mathrm{R}\Gamma(X_0(p), \mathbf{D}(\omega^k)) \\ 
\mathrm{R}\Gamma(X_0(p), \mathbf{D}(\omega^k)) \ar[r]^{q_2^\star} \ar[u]^{w^\star} & \mathrm{R}\Gamma(C , q_2^\star \mathbf{D}( \omega^k))  \ar[r] \ar[u]^{\iota^\star} & \mathrm{R}\Gamma(C , q_1^\star \mathbf{D}( \omega^k)) \ar[r]^{\mathrm{tr}_{q_1}} \ar[u]^{\iota^\star} & \mathrm{R}\Gamma(X_0(p), \mathbf{D}(\omega^k)) \ar[u]^{w^\star \langle p\rangle^{-1}}}
\]
We see that 
$\langle U^{\naive}_pf ,  w^\star g \rangle_0 = \langle  f ,w^\star \langle p\rangle^{-1} U^{\naive}_p g \rangle_0$.
We now check that the normalizing factors are correct so that $\langle U_pf ,  g \rangle = \langle  f ,\langle p\rangle^{-1}U_p g \rangle$.
\end{proof}

\subsubsection{The $p$-adic pairing} We now work again over $\oscr_{\mathcal{W}_r}$. We let $X_0(p)^{m, \dag} = \colim_{a \rightarrow 1} X_0(p)_{[a,1]}$. We let $X_0(p)^{\et, \dag} = \colim_{a \rightarrow 0} X_0(p)_{[0,a]}$.   The Atkin--Lehner map is an isomorphism $ w : X_0(p)^{m, \dag} \rightarrow X_0(p)^{\et, \dag}$ and there is an isomorphism $w : w^\star \omega^{\kappa^{\un}} \rightarrow \omega^{\kappa^{\un}}$ (compare with Sections \ref{sect-defining-U_p} and \ref{sect-defining-U_p2}). 

\begin{lem} We have a canonical perfect pairing
\[\langle -,- \rangle_0 : \HH^0\left(X_0(p)^{m, \dag}, \omega^{\kappa^{\un}}\right) \times \HH^1_c\left(X_0(p)^{m, \dag}, \omega^{2-\kappa^{\un}}(-D)\right) \longrightarrow \oscr_{\mathcal{W}_r}.\]
Moreover
\begin{itemize}
\item $\HH^0(X_0(p)^{m, \dag}, \omega^{\kappa^{\un}})$ is a compact inductive limit  of projective Banach spaces over $\oscr_{\mathcal{W}_r}$,
\item $\HH^1_c(X_0(p)^{m, \dag}, \omega^{2-\kappa^{\un}}(-D))$ is a compact projective limit of projective Banach spaces over $\oscr_{\mathcal{W}_r}$,
\end{itemize}
and both spaces are strong duals of each other. 
\end{lem}

\begin{proof} The pairing is obtained as follows: 
\[\HH^0\left(X_0(p)^{m, \dag}, \omega^{\kappa^{\un}}\right) \times \HH^1_c\left(X_0(p)^{m, \dag}, \omega^{2-\kappa^{\un}}(-D)\right) \longrightarrow   \HH^1_c\left(X_0(p)^{m, \dag}, \Omega^1_{X_0(p)^{m}/\qq_p} \hat{\otimes}_{\qq_p}\oscr_{\mathcal{W}_r}\right) \xrightarrow{\mathrm{res}_{X_0(p)^m}} \oscr_{\mathcal{W}_r}.\]
We prove the remaining claims. Let $\mathcal{W}^i$ be the connected component of the character $i : x \mapsto x^i$ for $i= 0, \ldots, p-2$ in $\mathcal{W}$. Then for all $i$, $\omega^{\kappa^{\un}}\vert_{\mathcal{W}_r^i} = \omega^{i} \hat{\otimes}_{\qq_p} \oscr_{\mathcal{W}^i_r}$ is an ``isotrivial'' sheaf. This follows from the existence of the Eisenstein family (see \cite{MR3097946}, the final discussion below Proposition 6.2). Therefore, $$\HH^0(X_0(p)^{m, \dag}, \omega^{\kappa^{\un}})\otimes_{\oscr_{\mathcal{W}_r}} \oscr_{\mathcal{W}_r^i} = \HH^0(X_0(p)^{m, \dag}, \omega^i) \hat{\otimes}_{\qq_p}  \oscr_{\mathcal{W}_r^i}$$
 $$\HH^1_c(X_0(p)^{m, \dag}, \omega^{\kappa^{\un}})\otimes_{\oscr_{\mathcal{W}_r}} \oscr_{\mathcal{W}_r^i} = \HH^1_c(X_0(p)^{m, \dag}, \omega^i) \hat{\otimes}_{\qq_p}  \oscr_{\mathcal{W}_r^i}$$ 
and similarly for cuspidal cohomology. 
All the statements of the lemma are reduced to the similar statements for the classical invertible sheaves $\omega^i$, and they follow from the usual duality for coherent cohomology of affinoid dagger spaces recalled in Section~\ref{sect-prelimi}.
\end{proof}
  
\medskip
  
We deduce form this lemma that there is a canonical pairing: 
\[\langle -,- \rangle : \HH^0\left(X_0(p)^{m, \dag}, \omega^{\kappa^{\un}}\right) \times \HH^1_c\left(X_0(p)^{\et, \dag}, \omega^{2-\kappa^{\un}}(-D)\right) \longrightarrow \oscr_{\mathcal{W}_r}\]
by putting $\langle -,- \rangle = \langle -,w^\star(-) \rangle_0.$ 
 For this pairing,  $\langle U_p (-)  , - \rangle = \langle - , \langle p\rangle^{-1}U_p (-)\rangle$.  We have by definition that $ \HH^0(X_0(p)^{m, \dag}, \omega^{\kappa^{\un}}) = \colim_{a \rightarrow 1}\HH^0(X_0(p)_{[a,1]}, \omega^{\kappa^{\un}})$. 
 
 \begin{lem} We have a canonical isomorphism:  
\[\HH^1_c\left(X_0(p)^{\et, \dag}, \omega^{2-\kappa^{\un}}(-D)\right) = \lim_{a \rightarrow 0} \HH^1_{X_0(p)_{[0,a[}}\left(X_0(p), \omega^{2-\kappa^{\un}}(-D)\right).\]
\end{lem}

\begin{proof} We have a short exact sequence for $0 <a < b$ small enough: 
\[
\begin{tikzcd}
0 \ar[r] &  \HH^0\left(X_0(p)_{[0,b]}, \omega^{2-\kappa^{\un}}(-D)\right) \ar[r] &  \HH^0\left(X_0(p)_{[a,b]}, \omega^{2-\kappa^{\un}}(-D)\right) \arrow[d, phantom, ""{coordinate, name=Z}] \arrow[d,
rounded corners,
to path={ -- ([xshift=2ex]\tikztostart.east)
|- (Z) [near start]\tikztonodes
-| ([xshift=-2ex]\tikztotarget.west)
-- (\tikztotarget)}] & \\
&& \HH^1_{X_0(p)_{[0,a[}}\left(X_0(p)_{[0,b]}, \omega^{2-\kappa^{\un}}(-D)\right) \ar[r] & 0.
\end{tikzcd}\]
Passing to the limit as $a \rightarrow 0$ proves the lemma.
\end{proof}

\bigskip

Therefore the operator $U_p$ is compact on both cohomology groups.  We deduce from the pairing that the characteristic series of $U_p$ in degree $0$ is the same as the characteristic series of $\langle p \rangle^{-1} U_p$ in degree $1$, so that $\mathcal{X} = \mathcal{Z}$. We have a canonical perfect pairing $\langle -,- \rangle  : \mathcal{M} \times \mathcal{N} \rightarrow \pi^{-1} \oscr_{\mathcal{W}_r}$, for which
$\langle  T_\ell f , g \rangle
 =\langle f, T_\ell^tg\rangle$ for any prime number $\ell$ not dividing $Np$. We recall that $T_\ell^t = \langle \ell \rangle^{-1} T_\ell$.   
 We deduce that the eigencurves $\mathcal{C}$ and $\mathcal{D}$ are  canonically isomorphic.


\end{document}